\documentclass{aic}

\usepackage{amsmath,amsthm,amssymb,hyperref, color, bm}
\usepackage{blkarray}

\newcommand{\remove}[1]{}

\makeatletter
\newtheorem*{rep@theorem}{\rep@title}
\newcommand{\newreptheorem}[2]{%
\newenvironment{rep#1}[1]{%
 \def\rep@title{#2 \ref{##1}}%
 \begin{rep@theorem}}%
 {\end{rep@theorem}}}
\makeatother

\newcommand\numberthis{\addtocounter{equation}{1}\tag{\theequation}}

\newtheorem{thm}{Theorem}[section]
\newreptheorem{thm}{Theorem}
\newtheorem{claim}[thm]{Claim}
\newtheorem{lem}[thm]{Lemma}
\newreptheorem{lem}{Lemma}
\newtheorem{define}[thm]{Definition}
\newtheorem{cor}[thm]{Corollary}
\newtheorem{obs}[thm]{Observation}
\newtheorem{stat}[thm]{Statement}

\newtheorem{comm}[thm]{Comment}

\newtheorem{conjecture}[thm]{Conjecture}

\def\GL{\text{GL}}
\def\F{{\mathbb{F}}}

\def\Q{{\mathbb{Q}}}
\def\Z{{\mathbb{Z}}}
\def\N{{\mathbb{N}}}

\def\C{{\mathbb{C}}}

\def\P{{\mathbb{P}}}

\def\cC{{\cal C}}

\def\cU{{\cal U}}
\def\cM{\mathcal{M}}

\def\Ind{{\mathbf{1}}}

\def\balpha{{\bm \alpha}}
\def\bbeta{{\bm\beta}}

\def\_{\,\,\,\,\,}

\def\rank{\textsf{rank}}
\def\crank{\textsf{crank}}

\def\Coeff{\text{Coeff}}

\newcommand{\eps}{\epsilon}
\usepackage{leftidx}

\usepackage{tikz-cd}
\aicAUTHORdetails{%
  title = {The Kakeya Set conjecture over $\mathbb{Z}/N\mathbb{Z}$ for general $N$}, %% please capitalize all significant words
  author = {Manik Dhar},
  plaintextauthor = {Manik Dhar},
  plaintexttitle = {The Kakeya Set conjecture over Z/N Z for general N}, %%  title without math or LaTeX
  keywords = {Kakeya conjecture, polynomial method},
}   %%% END \aicAUTHORdetails

\aicEDITORdetails{%
   year={2024},
   number={2},
   received={18 January 2022},   % received date: example: 7 January 2017
   revised={4 June 2023},    % Optional revised date (you may comment it out)
   published={26 January 2024},  % published date
   doi={10.19086/aic.2024.2},      % XXX = number of paper, e.g. aic006 for paper#6
}   %%% END \aicEDITORdetails

\begin{document}

\begin{frontmatter}[classification=text]
%% EDITOR: this will force the keywords to appear right after the Abstract.
%%   If the abstract is too long and would force the keywords off the
%%   front page, please comment out % [classification=text] above
%%   This way the keywords will be floated on the bottom of the first page
%%   even though the Abstract spills over to the next page.

%%% AUTHOR: Title goes here.  This line is optional.  You must use it
%%   if title has footnote attached or requires nontrivial typesetting,
%%   e.g., inclusion of linebreaks to force nice layout.
\title{The Kakeya Set conjecture over $\mathbb{Z}/N\mathbb{Z}$ for general $N$} %% please capitalize all significant words

%%% AUTHOR:
%%% List all authors. If you wish, place grant acknowledgements in \thanks.
%%% In brackets include a short tag for each author.
\author[md]{Manik Dhar}
% \thanks{Department of Mathematics, Massachusetts institute of Technology. Email: \texttt{dmanik@mit.edu}. }

%%% AUTHOR: Abstract goes here
\begin{abstract}
We prove the Kakeya set conjecture for $\mathbb{Z}/N\mathbb{Z}$ for general $N$ as stated by Hickman and Wright~\cite{hickman2018fourier}. This entails extending and combining the techniques of Arsovski~\cite{arsovski2021padic} for $N=p^k$ and the author and Dvir~\cite{dhar2021proof} for the case of square-free $N$. We also prove stronger lower bounds for the size of $(m,\epsilon)$-Kakeya sets over $\Z/p^k\Z$ by extending the techniques of \cite{arsovski2021padic} using multiplicities as was done in \cite{saraf2008,dvir2013extensions}. In addition, we show our bounds are almost sharp by providing a new construction for Kakeya sets over $\Z/p^k\Z$ and $\Z/N\Z$.
\end{abstract}
\end{frontmatter}

%%% AUTHOR: body of paper starts here
\section{Introduction}
We are interested in proving lower bounds for the sizes of sets in $(\Z/N\Z)^n$ which have large intersections with lines in many directions. We first define the set of possible directions a line can take in $(\Z/N\Z)^n$.

\begin{define}[Projective space $\P (\Z/N\Z)^{n-1}$]
Let $N=p_1^{k_1}\hdots p_r^{k_r}$ where $p_1,\hdots,p_r$ are distinct primes. The Projective space $\P (\Z/N\Z)^{n-1}$ consists of vectors $u\in (\Z/N\Z)^n$ up to unit multiples of each other such that {\em $u$ (mod $p_i^{k_i}$)} has at least one unit co-ordinate for every $i=1,\hdots,r$.
\end{define}

For each direction in $\P (\Z/N\Z)^{n-1}$ we pick a representative in $(\Z/N\Z)^n$. This allows us to treat $\P (\Z/N\Z)^{n-1}$ as a subset of $(\Z/N\Z)^n$.

\begin{define}[$m$-rich lines]
Let $N,n\in \N$. For a subset $S\subseteq (\Z/N\Z)^n$, we say a line $L\subseteq (\Z/N\Z)^n$ is $m$-rich with respect to $S$ if $|S\cap L|\ge m$.
\end{define}

We now define sets that have large intersections with lines in many directions.

\begin{define}[$(m,\epsilon)$-Kakeya Sets]
Let $n,N\in \N$. A set $S\subseteq (\Z/N\Z)^n$ is said to be $(m,\epsilon)$-Kakeya if for at least an $\epsilon$ fraction of directions $u\in \P (\Z/N\Z)^{n-1}$ there exists a line $L_u=\{a+\lambda u| \lambda \in \Z/N\Z\}$ in the direction $u$ which is $m$-rich with respect to $S$.

An $(N,1)$-Kakeya set in $(\Z/N\Z)^n$ is simply called a Kakeya set. In other words, a Kakeya set is a set that contains a line in every direction. 
\end{define}

In this paper, we resolve the following conjecture of Hickman and Wright~\cite{hickman2018fourier}.

\begin{conjecture}[Kakeya set conjecture over $\Z/N\Z$]\label{conj}
For all $\eps > 0$ and $n\in \N$, there exists a constant $C_{n,\eps}$ such that any Kakeya set $S \subset (\Z/N\Z)^n$ satisfies 
$$ |S| \geq C_{n,\epsilon} N^{n-\eps}.$$
\end{conjecture}

Wolff in \cite{wolf1999} first posed Conjecture~\ref{conj} with $\Z/N\Z$ replaced by a finite field as a possible problem whose resolution might help in proving the Euclidean Kakeya conjecture. Wolff's conjecture was proven by Dvir in \cite{dvir2009size} with $C_{n}=1/n!$ (Over finite fields of size $N$ the $\epsilon$ dependence in Conjecture~\ref{conj} is not needed. For composite $N$ it is known that the $\epsilon$ dependence is essential~\cite{hickman2018fourier,dhar2021proof}.) Using the method of multiplicities and its extensions \cite{saraf2008,dvir2013extensions,bukh2021sharp} the constant was improved to $C_{n}=2^{-n+1}$, which is known to be tight. 

Ellenberg, Oberlin, and Tao in \cite{ellenberg2010kakeya} proposed studying the size of Kakeya sets over the rings $\Z/p^k\Z$ and $\F_q[x]/\langle x^k\rangle$. They were motivated by the fact that these rings have `scales' and hence are closer to the Euclidean version of the problem. Hickman and Wright posed Conjecture~\ref{conj} for $\Z/N\Z$ with arbitrary $N$ and considered connections between the problem and the Kakeya conjecture over the $p$-adics in \cite{hickman2018fourier}. Indeed, resolving Conjecture~\ref{conj} for the rings $\Z/p^k\Z$ and $\F_q[x]/\langle x^k\rangle$ resolves the Minkowski dimension Kakeya set conjecture for the $p$-adic integers and the power series ring $\F_q[[x]]$ respectively~\cite{ellenberg2010kakeya,dummithablicsek2013,hickman2018fourier}. The Kakeya problem over these rings is interesting as, similarly to the Euclidean case, one can construct Kakeya sets of Haar measure $0$ for the $p$-adic integers and the power series ring $\F_q[[x]]$. These constructions and generalizations to other settings can be found in \cite{dummithablicsek2013,Fraser_2016,CML_2018__10_1_3_0,hickman2018fourier}.

For $n=2$ and $N=p^k$, Conjecture~\ref{conj} was resolved by Dummit and Hablicsek in \cite{dummithablicsek2013} by proving that sizes of Kakeya sets are lower bounded by $p^{2k}/2k$. The author and Dvir~\cite{dhar2021proof} resolved Conjecture~\ref{conj} for the case of $N$ square-free by proving the following bound, which implies Conjecture~\ref{conj} using Observation~\ref{obs:no} at the end of this section.

\begin{thm}[Kakeya set bounds for square-free $N$ {\cite[Theorem~1.3]{dhar2021proof}}]\label{squarefr}
Let $N\in \N$ where $N=p_1\hdots p_r$ for distinct primes $p_1,\hdots,p_r$. Any Kakeya set $S$ in $(\Z/N\Z)^n$ for $n\in \N$ satisfies,
$$|S|\ge 2^{-rn}N^n.$$
\end{thm}

In \cite{dhar2021proof} a reduction was also proven which lower bounds the size of Kakeya sets in $(\Z/p^k\Z)^n$ by the $\F_p$-rank of the point-hyperplane incidence matrix of $(\Z/p^k\Z)^n$. Building on the ideas behind this reduction, Arsovski proved Conjecture~\ref{conj} for $N=p^k$ for general $n$.

\begin{thm}[Kakeya Set bounds over $\Z/p^k\Z$,{\cite[Theorem~2]{arsovski2021padic}}]\label{power}
For $p$ prime and $k,n\in \N$, every Kakeya set $S$ in $(\Z/p^k\Z)^n$ satisfies,
$$|S|\ge (kn)^{-n} p^{kn}.$$
\end{thm}

As mentioned before this also resolves the Minkowski dimension Kakeya Set conjecture over the $p$-adics. In \cite{arsovskiNew} (which is the arxiv version 2 of the paper~\cite{arsovski2021padic}) a different approach inspired by the polynomial method proofs in \cite{dvir2009size,saraf2008,dvir2013extensions} is used to give the following bound for $(m,\epsilon)$-Kakeya sets.

\begin{thm}[$(m,\epsilon)$-Kakeya Set bounds over $\Z/p^k\Z$,{\cite[Theorem~3]{arsovskiNew}}]\label{power2}
For $p$ prime and $k,n\in \N$, every $(m,\epsilon)$-Kakeya set $S$ in $(\Z/p^k\Z)^n$ satisfies,
$$|S|\ge (Ckn^2p)^{-n}\epsilon^n m^{n},$$
where $C>1$ is some universal constant.
\end{thm}

As can be seen from the bounds in the Theorems \ref{power} and \ref{power2}, the analysis in \cite{arsovskiNew} is looser than the analysis in \cite{arsovski2021padic} for the case of $(p^k,1)$-Kakeya sets and leads to worse bounds in that setting.

In \cite{arsovski2021padic,arsovskiNew} tools from $p$-adic analysis are used to develop the techniques which prove Theorems \ref{power} and \ref{power2}. Both the proofs can be thought of as trying to develop a polynomial method on roots of unity in $\C$. Our first result extends the ideas in \cite{arsovski2021padic} using multiplicities as in \cite{saraf2008,dvir2013extensions} to bound the size of $(m,\epsilon)$-Kakeya sets over $\Z/p^k\Z$ with improved constants. Another advantage of our proof is that it is more elementary and does not require tools from $p$-adic analysis.

\begin{thm}[Stronger $(m,\epsilon)$-Kakeya Set bounds over $\Z/p^k\Z$\footnote{We also note that the techniques presented here can also be used to prove norm bounds for functions $f:(\Z/N\Z)^n\rightarrow \C$ which have rich lines in many directions as was done for $N$ prime (and in general for finite fields) in Theorem 19 of \cite{dhar2019simple}. A line $L$ is $m$-rich with respect to $f$ in this setting if $\sum_{x\in L} |f(x)|\ge m$.  }]\label{thm:Smain}
Let $k,n,p\in \N$ with $p$ prime. Any $(m,\epsilon)$-Kakeya set $S\subseteq (\Z/p^k\Z)^n$ satisfies the following bound,
$$|S| \ge \epsilon\cdot\left(\frac{m^n}{(2(k+\lceil\log_p(n)\rceil))^n}\right).$$
When $p>n$, we also get the following stronger bound for $(m,\epsilon)$-Kakeya set $S$ in $(\Z/p^k\Z)^n$,
$$|S|\ge \epsilon\cdot\left(\frac{m^n}{(k+1)^n}\right) (1+n/p)^{-n}.$$
\end{thm}

We note that for fields of prime size, the bound above recovers the result of \cite{dvir2013extensions} as $p$ tends to infinity. This gives us a new proof for the result with new techniques. 

Combining techniques developed for $\Z/p^k\Z$ with techniques from \cite{dhar2021proof}, we prove the following lower bound for sizes of Kakeya sets in $\Z/N\Z$ for general $N$. 

\begin{thm}[Kakeya set bounds for $\Z/N\Z$]\label{main}
Let $n\in \N$ and $R=\Z/N\Z$  where $N=p_1^{k_1}\hdots p_r^{k_r}$ with distinct primes $p_1,\hdots,p_r$ and $k_1,\hdots,k_r\in \N$. Any Kakeya set $S$ in $R^n$ satisfies,
$$|S|\ge \left(\prod\limits_{i=1}^r (2(k_i+\lceil \log_{p_i}(n)\rceil))^{-n}\right)\cdot N^n.$$

When $p_1,\hdots,p_r\ge n$, we also get the following stronger lower bound for the size of a Kakeya set $S$ in $(\Z/N\Z)^n,N=p_1^{k_1}\hdots p_r^{k_r}$,
$$|S|\ge N^n\prod\limits_{i=1}^r(k_i+1)^{-n}(1+n/p_i)^{-n}.$$
\end{thm}

The bound above recovers Theorem \ref{squarefr} as the size of the divisors $p_i$ grows towards infinity.

Note, the techniques here do not naively prove $(m,\epsilon)$-Kakeya bounds over $\Z/N\Z$ for general $N$ as the techniques in \cite{dhar2021proof} also do not seem to extend to this setting (even for square-free $N$). In a follow-up paper~\cite{dhar2022maximal} we solve this problem by using probabilistic arguments on top of the techniques in this paper and also give Maximal Kakeya bounds over $\Z/N\Z$ for general $N$.

\begin{obs}\label{obs:no}
The number of prime divisors of $N$ satisfies $r = O( \log N / \log \log N)$. The expression $\prod_{i=1}^{r} k_i$ is upper bounded by the number of divisors $\tau(N)$ of $N$ which satisfies $\log(\tau(N))=O( \log N / \log \log N)$. The proof for the bound on $r$ can be found in \cite{zbMATH02610444} and the bound for $\tau(n)$ is {\em Theorem 317} in \cite{Hardy2008AnIT}. We now see that the expression in Theorem~\ref{main} is lower bounded by $C_n N^{n - O(n/\log \log N)}$ and so it indeed proves Conjecture~\ref{conj} for all $N$. \end{obs}

Note, we could also get Kakeya Set bounds for $\Z/N\Z$ for general $N$ by combining the techniques from \cite{arsovski2021padic} and \cite{dhar2021proof}. This would also resolve Conjecture~\ref{conj} by proving a Kakeya set lower bound of $C_n N^{n - O(n\log(n)/\log \log N)}$ with a worse dependence on the dimension $n$.

We also construct Kakeya sets over $\Z/p^k\Z$ showing that the bound in Theorem~\ref{thm:Smain} is close to being sharp. 

\begin{thm}[Small Kakeya sets over $\Z/p^k\Z$]\label{thm:con}
Let $s,n\in \N$, $p\ge 3$ be a prime and $k=(p^{s+1}-1)/(p-1)$. There exists a Kakeya set $S$ in $(\Z/p^k\Z)^n$ such that,
$$|S|\le \frac{p^{kn}}{k^{n-1}}(1-1/p)^{-n}.$$
\end{thm}

The construction here uses ideas from the earlier constructions in \cite{hickman2018fourier,Fraser_2016,CML_2018__10_1_3_0} but is quantitatively stronger.

Using the Chinese remainder theorem we also get a construction for $N$ with multiple prime factors showing that the bounds in Theorem~\ref{main} are also close to being sharp.

\begin{cor}[Small Kakeya sets over $\Z/N\Z$]\label{cor:con}
Let $n,s_1,\hdots,s_r\in \N$, $p_1,\hdots,p_r\ge 3$ be primes, $k_i=(p_i^{s_i+1}-1)/(p_i-1),i=1,\hdots,r$ and $N=p_1^{k_1}\hdots p_r^{k_r}$. There exists a Kakeya set $S$ in $(\Z/N\Z)^n$ such that,
$$|S|\le \frac{N^n}{k_1^{n-1}\hdots k_r^{n-1}} \prod\limits_{i=1}^r (1-1/p_i)^{-n} $$
\end{cor}
\begin{proof}
Using the Chinese Remainder Theorem (see Lemma~\ref{fact:remainder}) we see that the product of Kakeya sets over $\Z/p_i^{k_i}\Z$ for $i=1,\hdots,r$ will be a Kakeya set over $\Z/N\Z$. We are then done using Theorem~\ref{thm:con}.
\end{proof}

\subsection{Proof Overview}\label{sec:proofOver}
We first start with a brief overview of the approach introduced in \cite{dhar2021proof}. Given a Kakeya set $S$ in $(\Z/N\Z)^n$, we construct a matrix $K_S$ whose columns are indexed by points in $(\Z/N\Z)^n$ and its rows are indexed by a direction in $\P (\Z/N\Z)^{n-1}$ where the $d$th row would be supported on the line in direction $d$ contained in $S$. This ensures that the non-zero columns of $K_S$ correspond to points in $S$, which implies that the rank of $K_S$ then lower bounds the size of $S$. The goal is to construct a suitable $K_S$ and find a matrix $E$ such that $K_SE$ is a matrix independent of $S$ whose rank can be analyzed easily. For prime $N=p$, $E$ can be a matrix whose columns contain the evaluation of a monomial on $\F_p^n$. In this case $K_S$ ends up being a `decoder' matrix where each row outputs the evaluation of a monomial on a direction given its evaluations on a line in that direction. This turns out to be a reformulation of Dvir's polynomial method proof~\cite{dvir2009size}. When $N=pq$ (or in general is square-free) as $\Z/N\Z\cong \F_p\times \F_q$ we define $E$ as a tensor of matrices one acting on the $\F_p$ part and the other on the $\F_q$ part. This allows for an inductive argument to give Kakeya set lower bounds for square-free $N$. To get quantitatively stronger bounds we can use the evaluations of (Hasse) derivatives and multiplicities as was done in \cite{saraf2008,dvir2013extensions}.

In \cite{arsovski2021padic} for prime power $N$ points in $\F_p^n$ are embedded in $\C_p^n$ (the $p$-adic complex numbers which is isomorphic to $\C$ as a field) using $p^k$th roots of unity. The proof then implements the strategy of \cite{dhar2021proof} using matrices with polynomial entries. $U$ is a matrix whose columns are the evaluations of monomials over the embedding of $(\Z/p^k\Z)^n$ in the torus. Let $L$ be a line in direction $d$. The key statement is that some linear combination (with polynomial coefficients) of the rows in $U$ corresponding to points in $L$ can generate the $d$th row of a `Vandermonde' matrix $M$ after applying a mod $p$ operation (using the $p$-adic structure). The rank of $M$ can be easily lower bounded. The linear combination being taken here corresponds to a row of $K_S$ for a Kakeya set $S\subseteq (\Z/p^k\Z)^n$ containing $L$ in direction $d$. The key statement can be reformulated as saying that 

\begin{equation}\label{eq:proofOverTorusPoly}
    K_SU  \mod p = M
\end{equation}

The argument can be completed by saying that the non-zero columns of $S$ correspond to points in $S$, alternatively we note that we only use rows in $E$ which correspond to points in $S$. This shows the rank of $M$ over $\F_p$ lower bounds the size of $S$. The argument in \cite{arsovski2021padic} is only for $(p^k,1)$-Kakeya sets. This is because, apriori it is not clear how to define a suitable $K_S$ (equivalently how to decode from lines with only some points contained in $S$). The revised version of the paper~\cite{arsovskiNew} has a very different approach to give $(m,\epsilon)$-Kakeya bounds for the prime power case (the bounds there are quantitatively weaker than the ones here).

We now discuss how we improve the prime power bound quantitatively. Using simple linear algebraic arguments and the Chinese remainder theorem for a suitable polynomial ring we show that given $m$ points on a line $L$ in direction $d$ we can decode the $d$th row of $M$ up to degree $m$ (see Corollary~\ref{lem:decoding}). We use multiplicities and (Hasse) derivatives to get stronger quantitative bounds (which Theorem~\ref{thm:con} shows are almost sharp). We also develop these ideas without using the theory of $p$-adic numbers. We make this precise in Section~\ref{sec:PreAr}.

To get bounds for general $N$ we follow an inductive style argument as was done in \cite{dhar2021proof} with some small technical improvements to get better constants.

\subsection{Organization of the paper}
In Section \ref{sec:PreDD} we state definitions and results we need from \cite{dhar2021proof}. In Section \ref{sec:PreAr} we state and extend results from \cite{arsovski2021padic}. In Section \ref{sec:pkBd} we prove Theorem~\ref{thm:Smain}. In Section \ref{sec:Nbd} we prove Theorem~\ref{main}. Finally, in section 6 we prove Theorem~\ref{thm:con}.

\subsection{Acknowledgements}
The author would like to thank Zeev Dvir for helpful comments and discussion. The author would also like to thank the reviewers for their helpful suggestions. This work was done while the author was supported by the NSF grant DMS-1953807.

\section{Rank and crank of matrices with polynomial entries}\label{sec:PreDD}

\begin{define}[Rank of matrices with entries in {$\mathbb{F}[z]/ \langle f(z)\rangle$}]
Given a field $\F$ and a matrix $M$ with entries in $\F[z]/\langle f(z)\rangle$ where $f(z)$ is a non-constant polynomial in $\F[z]$, we define the $\F$-rank of $M$ denoted {\em $\rank_\F M$}  as the maximum number of $\F$-linearly independent columns of $M$.
\end{define}

In our proof, it will also be convenient to work with the following extension of rank for sets of matrices.

\begin{define}[crank of a set of matrices]
Let $\F$ be a field and $f(z)$ a non-constant polynomial in $\F[z]$. Given a finite set $T=\{A_1,\hdots,A_n\}$ of matrices over the ring $\F[z]/\langle f(z)\rangle$ having the same number of columns we let {\em $\crank_\F(T)$} be the $\F$-rank of the matrix obtained by concatenating all the elements $A_i$ in $T$ along their columns.\footnote{
In \cite{dhar2021proof} crank was defined using the rows of the matrices. The two definitions are equivalent when the entries of the matrices are from the field $\F$ following from simple linear algebra. In our setting, the column and row definitions are not necessarily identical and the column version is better for our purposes.} 
\end{define}

We will use a simple lemma which follows from the definition of crank.

\begin{lem}[crank bound for multiplying matrices]\label{lem:crankMatMult}
Let $\F$ be a field and $f(z)$ a non-constant polynomial in $\F[z]$. Given matrices $A_1,\hdots,A_n$ of size $a\times b$ and matrices $H_1,\hdots, H_n$ of size $c\times a$ with entries in $\F[x]/\langle f(z)\rangle$ we have
\em{$$\crank_\F\{A_i\}_{i=1}^n\ge\crank_\F\{H_i \cdot A_i\}_{i=1}^n.$$}
\end{lem}
\begin{proof}
A dependence on the columns of the matrix obtained by concatenating $A_i$s will be represented by a non-zero vector $w\in \F^b$ such that $A_iw=0$ for all $i=1,\hdots,n$. This would imply $(H_i\cdot A_i)w=0$ for all $i$ as well, completing the proof.
\end{proof}

Given a matrix $A$ with entries in $\F[z]/\langle f(z)\rangle$ we want to construct a new matrix with entries only in $\F$ such that their $\F$-ranks are the same. First, we state a simple fact about $\F[z]/\langle f(z)\rangle$.

\begin{lem}[Unique representation of elements in $\F [z\text{]}/\langle f(z)\rangle$]
Let $\F$ be a field and $f(z)$ a non-constant polynomial in $\F[z]$ of degree $d>0$. Every element in $\F[z]/\langle f(z)\rangle$ is uniquely represented by a polynomial in $\F[z]$ with degree strictly less than $d$ and conversely every degree strictly less than $d$ polynomial in $\F[z]$ is a unique element in $\F[z]/\langle f(z)\rangle$. When we refer to an element $h(z)\in \F[z]/\langle f(z)\rangle$ we also let it refer to the unique degree strictly less than $d$ polynomial it equals.
\end{lem}

\begin{define}[Coefficient matrix of $A$]
Let $\F$ be a field and $f(z)$ a non-constant polynomial in $\F[z]$ of degree $d>0$. Given any matrix $A$ of size $n_1\times n_2$ with entries in $\F[z]/\langle f(z)\rangle$ we can construct the {\em coefficient matrix of $A$} denoted by {\em $\Coeff(A)$} with entries in $\F$ which will be of size $dn_1\times n_2$ whose rows are labeled by elements in $\{0,\hdots,d-1\}\times [n_1]$ such that its $(i,j)$'th row is formed by the coefficients of $z^i$ of the polynomial entries of the $j$'th row of $A$. 
\end{define}

The key property of the coefficient matrix immediately follows from its definition.

\begin{lem}\label{lem:coeffR}
Let $\F$ be a field and $f(z)$ a non-constant polynomial in $\F[z]$ of degree $d>0$. Given any matrix $A$ with entries in $\F[z]/\langle f(z)\rangle$ and its coefficient matrix {\em $\Coeff(A)$} it is the case that an $\F$-linear combination of a subset of columns of $A$ is $0$ if and only if the corresponding $\F$-linear combination of the same subset of columns of {\em $\Coeff(A)$} is also $0$.

In particular, the $\F$-rank of $A$ equals the $\F$-rank of {\em $\Coeff(A)$}.
\end{lem}

We now need some simple properties related to the crank of tensor products. To that end, we first define the tensor/Kronecker product of matrices. Let $[r]=\{1,\hdots,r\}$.

\begin{define}[Kronecker Product of two matrices]
Given a commutative ring $R$ and two matrices $M_A$ and $M_B$ of sizes $n_1\times m_1$ and $n_2\times m_2$ corresponding to $R$-linear maps $A:R^{n_1}\rightarrow R^{m_1}$ and $B:R^{n_2}\rightarrow R^{m_2}$ respectively, we define the Kronecker product $M_A\otimes M_B$ as a matrix of size $n_1n_2\times m_1m_2$ with its rows indexed by elements in $[n_1]\times [n_2]$ and its columns indexed by elements in $[m_1]\times [m_2]$ such that
$$M_A\otimes M_B ((r_1,r_2),(c_1,c_2))=M_A(r_1,c_1)M_B(r_2,c_2),$$
where $r_1\in [n_1],r_2\in [n_2],c_1\in [m_1]$ and $c_2\in [m_2]$. $M_A\otimes M_B$ corresponds to the matrix of the $R$-linear map $A\otimes B: R^{n_1}\otimes R^{n_2}\cong R^{n_1n_2}\rightarrow R^{m_1}\otimes R^{m_2}\cong R^{m_1m_2}$.
\end{define}

We will need the following simple property of Kronecker products which follows from the corresponding property of the tensor product of linear maps.

\begin{lem}[Multiplication of Kronecker products]\label{fact:multOfKronecker}
Given matrices $A_1,A_2,B_1$ and $B_2$ of sizes $a_1\times n_1$, $a_2\times n_2$, $n_1\times b_1$ and $n_2\times b_2$ we have the following identity,
$$(A_1 \otimes A_2) \cdot (B_1\otimes B_2)=(A_1\cdot B_1)\otimes (A_2\cdot B_2).$$
\end{lem}

We want to prove a crank bound for a family of matrices with a tensor product structure. This statement is an analog of Lemma 4.8 in \cite{dhar2021proof} for our setting. Indeed we will prove it using Lemma 4.8 in \cite{dhar2021proof} which we now state.

\begin{lem}[Lemma 4.8 in \cite{dhar2021proof}]\label{lem:crankTensor}
Given matrices $A_1,\hdots,A_n$ of size $a_1\times a_2$ over a field $\F$ such that 
{\em $$\crank\{A_i\}_{i=1}^n\ge r_1$$} 
and matrices $B_{i,j}$ over the field $\F$ for $i\in [n]$ and $j\in [m]$ of size $b_1\times b_2$ such that 
{\em $$\crank\{B_{i,j}\}_{j=1}^m\ge r_2$$}
for all $i\in [n]$ we have,
\em{$$\crank\{A_i\otimes B_{i,j}|i\in [n],j\in[m]\}\ge r_1r_2.$$}
\end{lem}

The above lemma follows easily from simple properties of the tensor product and a proof can be found in \cite{dhar2021proof}. We now prove a generalization for our setting. 

\begin{lem}[crank bound for tensor products]\label{lem:crankTensorProduct}
Let $\F$ be a field and $f(z)\in \F[z]$ a non-constant polynomial. Given matrices $A_1,\hdots,A_n$ of size $a_1\times a_2$ over the ring $\F[z]/\langle f(z)\rangle$ such that 
{\em$$\crank_\F\{A_i\}_{i=1}^n\ge r_1$$} 
and matrices $B_{i,j}$ over the field $\F$ for $i\in [n]$ and $j\in [m]$ of size $b_1\times b_2$ such that 
{\em $$\crank_\F\{B_{i,j}\}_{j=1}^m\ge r_2$$}
for all $i\in [n]$ we have,
\em{$$\crank_\F\{A_i\otimes B_{i,j}|i\in [n],j\in[m]\}\ge r_1r_2.$$}
\end{lem}
Note the asymmetry between the matrices $A_i$ having entries in $\F[z]/\langle f(z)\rangle$ and the matrices $B_{i,j}$ only having entries in $\F$. This is important for the proof to work.
\begin{proof}
Consider the coefficient matrices $\Coeff(A_i)$. Using Lemma~\ref{lem:coeffR} it follows that 
\begin{align}\label{eq:T1}
    \crank_\F\{A_i\}_{i=1}^n=\crank_\F\{\Coeff(A_i)\}_{i=1}^n.
\end{align}

As $B_{i,j}$ only has entries in $\F$ it is easy to see that $\Coeff(A_i)\otimes B_{i,j}=\Coeff(A_i\otimes B_{i,j})$ (Note this would not be true if $B_{i,j}$ had entries in the ring $\F[z]/\langle f(z)\rangle$). Using Lemma~\ref{lem:coeffR} now gives us,

\begin{align}\label{eq:T2}
    \crank_\F\{A_i\otimes B_{i,j}|i\in [n],j\in[m]\}=\crank_\F\{\Coeff(A_i)\otimes B_{i,j}|i\in [n],j\in[m]\}.
\end{align}

We apply Lemma~\ref{lem:crankTensor} on the family $\Coeff(A_i)\otimes B_{i,j},i\in [n],j\in[m]$ and use equations \eqref{eq:T1} and \eqref{eq:T2} to complete the proof.
\end{proof}

\section{Polynomial Method on the complex torus}\label{sec:PreAr}
As mentioned in the introduction, the techniques presented here are an extension of the techniques used in \cite{arsovski2021padic} and are developed without invoking the tools and the language of $p$-adic analysis.

We embed $(\Z/p^k\Z)^n$ in the complex torus. As mentioned in the proof overview (Section~\ref{sec:proofOver}) we will be working with matrices with polynomial entries. We first define the rings where the entries come from.

\begin{define}[Rings $\overline{T}_{\ell}$ and $T_{\ell}^k$]
Let 
$$\overline{T}_{\ell}=\F_p[z]/\langle z^{p^\ell}-1\rangle$$ 
and 
$$T_{\ell}^k=\Z(\zeta_{p^k})[z]/\langle z^{p^\ell} -1\rangle$$
where $\zeta_{p^k}$ is a primitive complex $p^k$'th root of unity.
\end{define}

As stated in the proof overview (Section~\ref{sec:proofOver}) the goal is decode the $d$th row of a `Vandermonde' matrix starting from the evaluations of monomials on a line $L$ in direction $d$. We set $\ell$ according to the number of points in $L$ and to what order derivatives we are working with and it controls the size of our `Vandermonde' matrix as well. If we were not using any multiplicities/derivatives then $\ell$ could be set to $\log_p(m)$ to lower bound the size of $(m,\epsilon)$-Kakeya sets with weaker constants.

We suppress $p$ in the notation as it will be fixed to a single value throughout our proofs. We also let $\zeta=\zeta_{p^k}$ throughout this section for ease of notation.

Note that $\Z(\zeta)$ is the ring $\Z[x]/\langle \phi_{p^k}(x)\rangle$ where,
\begin{align}
    \phi_{p^k}(x)=\frac{x^{p^k}-1}{x^{p^{k-1}}-1}=\sum\limits_{i=0}^{p-1} x^{ip^{k-1}}\label{eq:cyclo},
\end{align}
is the $p^k$ Cyclotomic polynomial. We will also work with the field $\Q(\zeta)=\Q[x]/\langle \phi_{p^k}(x)\rangle$. 

We need a simple lemma connecting $\Z(\zeta)$ to $\F_p$.

\begin{lem}[Quotient map $\psi_{p^k}$ from $\Z(\zeta)$ to $\F_p$]
The field $\F_p$ is isomorphic to $\Z(\zeta)/\langle p,\zeta-1\rangle$. In particular, the map $\psi_{p^k}$ from $\Z(\zeta)$ to $\F_p$ which maps $\Z$ to $\F_p$ via the mod $p$ map and $\zeta$ to $1$ is a ring homomorphism.
\end{lem}
\begin{proof}
$\zeta$ is a root of the Cyclotomic polynomial $\phi_{p^k}$ defined in \eqref{eq:cyclo}.
As mentioned earlier, this allows us to write $\Z(\zeta)$ as $\Z[\zeta]/\langle \phi_{p^k}(\zeta)\rangle$. Notice that using \eqref{eq:cyclo} we have $\phi_{p^k}(1)\text{ (mod }p\text{)}=0 $ which implies that $\phi_{p^k}(\zeta)$ is divisible by $\zeta-1$ over $\F_p$.

If we quotient the ring $\Z(\zeta)$ by $\langle p,\zeta-1\rangle$ we get 
$$\frac{\Z[\zeta]}{\langle \zeta-1,p, \phi_{p^k}(\zeta)\rangle}=\frac{\F_p[\zeta]}{\langle \zeta-1, \phi_{p^k}(\zeta)\rangle}=\frac{\F_p[\zeta]}{\langle \zeta-1\rangle}=\F_p.$$
Therefore, the map $\psi_{p^k}$ is the map which quotients the ring $\Z(\zeta)$ by the ideal $\langle p, \zeta-1 \rangle$.
\end{proof}

We note $\psi_{p^k}$ can be extended to the rings $\Z(\zeta)[z]/\langle h(z)\rangle$ for any $h(x)\in \Z(\zeta)[x]$ by mapping $z$ to $z$. The proof above immediately generalizes to the following corollary.

\begin{cor}[Extending $\psi_{p^k}$]\label{cor:psiE}
Let $h(x)\in \Z(\zeta)[x]$. $\psi_{p^k}$ is a ring homomorphism from $\Z(\zeta)[z]/\langle h(z)\rangle$ to $\Z(\zeta)[z]/\langle h(z), p,\zeta-1\rangle=\F_p[z]/\langle \psi_{p^k}(h(z))\rangle$.
\end{cor}
$T_{\ell}^k/\langle p,\zeta-1\rangle$ being isomorphic to $\overline{T}_{\ell}$ is a special case of the corollary. The $\psi_{p^k}$ will correspond to the mod $p$ operation from the overview.

The operation $\psi_{p^k}$ corresponds to the mod $p$ operation in \eqref{eq:proofOverTorusPoly} in the proof overview (Section~\ref{sec:proofOver}).

We now prove that the rank of a matrix can only decrease under the quotient map $\psi_{p^k}$.

\begin{lem}\label{lem:quoRank}
If $A$ is a matrix with entries in $T_{\ell}^k$, then we have the following bound,
{\em$$\rank_{\Q(\zeta)} A  \ge \rank_{\F_p} \psi_{p^k}(A),$$}
where $\psi_{p^k}(A)$ is the matrix with entries in $\overline{T}_{\ell}$ obtained by applying $\psi_{p^k}$ to each entry of $A$.  
\end{lem}
In \cite{arsovski2021padic} the above lemma is implicit in the proof of Proposition 4, Arsovski's proof uses tools from $p$-adic analysis. We provide an alternate proof without using those techniques. 
\begin{proof}[Proof of Lemma~\ref{lem:quoRank}]
Let $A$ be a matrix with entries in $T_{\ell}^k=\Z(\zeta)[z]/\langle z^{p^\ell}-1\rangle$. $T^k_\ell$ is a sub-ring of the ring $\Q(\zeta)[z]/\langle z^{p^\ell}-1\rangle$. As $\Q(\zeta)$ is a field we can define $\Coeff(A)$ with entries in $\Q(\zeta)$. 

This also means $\psi_{p^k}$ can be applied on $\Coeff(A)$. On the other hand $\psi_{p^k}$ can be directly applied on $A$ as it has entries in $T_\ell^k$ and $\psi_{p^k}(A)$ will have entries in $\overline{T}_\ell = \F_p[z]/\langle z^\ell -1 \rangle$. This means we can define $\Coeff(\psi_{p^k}(A))$ with entries in $\F_p$. Notice that $\psi_{p^k}(\Coeff(A))=\Coeff(\psi_{p^k}(A))$.

Let $\rank_{\Q(\zeta)} A=r$. By Lemma~\ref{lem:coeffR} and simple properties of the determinant, we have that every $r+1\times r+1$ sub-matrix $M$ of $\Coeff(A)$ has $0$ determinant. As $\psi_{p^k}$ is a ring homomorphism it immediately follows that the corresponding $r+1\times r+1$ sub-matrix $\psi_{p^k}(M)$ in $\psi_{p^k}(\Coeff(A))=\Coeff(\psi_{p^k}(A))$ will have $0$ determinant too. 

This implies that $\Coeff(\psi_{p^k}(A))$ has rank at most $r$. The statement now follows from applying Lemma~\ref{lem:coeffR} on $\psi_{p^k}(A)$.  
\end{proof}

We now define the `Vandermonde' matrix (corresponding to $M$ in Section~\ref{sec:proofOver}) with entries in $\overline{T}_{\ell}$ which was defined by Arsovski in \cite{arsovski2021padic} and whose $\F_p$-rank will help us lower bound the size of Kakeya Sets.

\begin{define}[The matrix $M_{p^\ell,n}$]
The matrix $M_{p^\ell,n}$ is a matrix over $\overline{T}_{\ell}$ with its rows and columns indexed by points in $(\Z/p^\ell\Z)^n$. The $(u,v)\in (\Z/p^\ell\Z)^n\times (\Z/p^\ell\Z)^n$ entry is
$$M_{p^\ell,n}(u,v)=z^{\langle u, v\rangle}.$$
\end{define}

\begin{lem}[Rank of $M_{p^\ell,n}$]\label{lem:1rank}
The $\F_p$-rank of $M_{p^\ell,n}$ is at least 
{\em $$\text{rank}_{\F_p} M_{p^\ell,n} \ge \binom{\lceil p^\ell \ell^{-1} \rceil +n}{n}.$$}
\end{lem}
A lower bound of $p^{\ell n} (\ell n)^{-n}$ is proven within Lemma 5 and Theorem 2 in \cite{arsovski2021padic}. The same argument can also give the bound above with a slightly more careful analysis. For completeness, we give a proof in Appendix \ref{ap:rk}. For convenience, we now prove that removing the rows in $M_{p^\ell,n}$ corresponding to elements $u\in (\Z/p^\ell\Z)^n\setminus\P (\Z/p^\ell\Z)^{n-1}$ does not change the $\F_p$-rank of the matrix. 

For a given set of elements $V\subseteq (\Z/p^\ell\Z)^{n}$ let $M_{p^\ell,n}(V)$ refer to the sub-matrix obtained by restricting to rows of $M_{p^\ell,n}$ corresponding to elements in $V$. In particular, for any given $u\in R^n$ we let $M_{p^k,n}(u)$ refer to the $u$'th row of the matrix.

\begin{lem}\label{lem:resRk}
{\em $$\text{rank}_{\F_p} M_{p^\ell,n}(\P (\Z/p^\ell\Z)^{n-1}) = \text{rank}_{\F_p} M_{p^\ell,n}.$$}
\end{lem}
\begin{proof}
Consider the matrix $\Coeff(M_{p^\ell,n})$. For any $u\in (\Z/p^\ell\Z)^n$, the $u$'th row in $M_{p^\ell,n}$ will correspond to a $p^\ell$ block of rows $B_u$ in $\Coeff(M_{p^\ell,n})$. 
Say $u$ doesn't have a unit coordinate. We can find an element $u' \in \P (\Z/p^\ell \Z)^{n-1}$ such that for some $i$, $p^iu'=u$. We claim that the block $B_u$ can be generated by the block $B_{u'}$ via a linear map. This follows because given the coefficient vector of a polynomial $f(z)$ in $\overline{T}_{\ell}$ the coefficient vector of the polynomial $f(z^p)\in \overline{T}_\ell$ can be obtained via a $\F_p$-linear map. This shows that the span of the rows of $\Coeff(M_{p^\ell,n})$ is identical to the span of the rows in $\Coeff(M_{p^\ell,n}(V)$ where $V$ is the set of vectors in $(\Z/p^k\Z)^n$ with at least one unit co-ordinate.

Now for any element $u\in V$ we can find an element $u'\in \P (\Z/p^\ell \Z)^{n-1}$ such that there exists a $\lambda \in (\Z/p^k\Z)^\times$ for which $u=\lambda u'$. We now note that the block $B_u$ can be generated by the block $B_{u'}$ via a linear map as the coefficient vector of a polynomial $f(z^{\lambda})\in \overline{T}_{\ell}$ can be linearly computed from $f(z)\in \overline{T}_{\ell}$. This shows that the span of the rows of $\Coeff(M_{p^\ell,n}(V))$ is identical to the span of the rows in $\Coeff(M_{p^\ell,n}(\P (\Z/p^k\Z)^{n-1})$ completing the proof using Lemma~\ref{lem:coeffR}.
\end{proof}

In \cite{arsovski2021padic} a row vector which encodes the evaluation of monomials over points in a line in direction $u\in \P (\Z/p^k\Z)^{n-1}$ is used to decode the $u$'th row of $M_{p^k,n}$. We extend this method by using row vectors which encode the evaluations of Hasse derivative of a monomial over a given point. We first define Hasse derivative.

\begin{define}[Hasse Derivatives]
Given a polynomial $f\in \F[x_1,\hdots,x_n]$ for any field $\F$ and an $\balpha\in \Z_{\ge 0}^n$ the $\balpha$'th {\em Hasse derivative} of $f$ is the polynomial $f^{(\balpha)}$ in the expansion $f(x+z)=\sum_{\bbeta\in \Z_{\ge 0}^n} f^{(\bbeta)}(x)z^{\bbeta}$ where $x=(x_1,...,x_n)$, $z=(z_1,...,z_n)$ and $z^{\bbeta}=\prod_{k=1}^n z_k^{\beta_k}$.  
\end{define}

\begin{define}[The evaluation vector $U_{d}^{(\balpha)}(y)$]
Let $\F$ be a field, $n,d\in \Z$ and $\balpha\in \Z_{\ge 0}^n$. For any given point $y\in \F^n$ we define $U_d^{(\balpha)}(y)$ to be a row vector of size $d^n$ whose columns are indexed by monomials $m=x_1^{j_1}x_2^{j_2}\hdots x_n^{j_n}\in \F[x_1,\hdots,x_n]$ for $j_k\in \{0,\hdots,d-1\},k\in [n]$ such that its $m$'th column is $m^{(\balpha)}(y)$. 
\end{define}

We suppress $n$ in the notation as there will be no ambiguity over it in this paper. We define the weight of $\balpha\in \Z_{\ge 0}^n$ as  $\text{wt}(\balpha)=\sum_{k=1}^n \alpha_k$. The following simple fact about uni-variate polynomials illustrates one use of Hasse derivatives and how they correspond to our intuition with regular derivatives in fields of characteristic $0$.

\begin{lem}\label{fact:Hasse}
Let $\F$ be a field, $y\in \F$ and $w\in \N$. For any uni-variate polynomial $f\in\F[x]$ all its Hasse derivatives at the point $y$ of weight strictly less than $w$ are $0$ if and only if $f\in \langle (x-y)^w\rangle$ or in other words $f$ is divisible by $(x-y)^w$.
\end{lem}

Let $L\subseteq (\Z/p^k\Z)^n$ be a line in direction $u\in \P (\Z/p^k\Z)^{n-1}$. A special case of the next lemma proves that, for any polynomial $f\in \Z[x_1,\hdots,x_n]$, we can evaluate $f(z^{u})\in \overline{T}_k=\F_p[z]/\langle z^{p^k}-1\rangle,z^{u}=(z^{u_1},\hdots,z^{u_n})$ from the evaluation of $f$ on the points $f(\zeta^x),x\in L$. This is implicit in the proof of Proposition 4 in \cite{arsovski2021padic}. We generalize this statement. Let $\pi: L \rightarrow \Z_{\ge 0}$ be a function on the line such that $\sum_{x\in L}\pi(x)\ge p^\ell$. We prove that we can decode $f(z^{u'})\in \overline{T}_{\ell}=\F_p[z]/\langle z^{p^\ell}-1\rangle$ from the evaluations of weight at most $\pi(x)$ Hasse derivatives of $f$ at $\zeta^x,x\in L$ for any $u'$ in $(\Z/p^\ell\Z)^n$ such that $u'\text{ (mod }p^k\text{)}=u$.

The lemma below can be thought of as analogous to how over finite fields the evaluation of a polynomial and its Hasse derivatives with high enough weight along a line in direction $u$ can be used to decode the evaluation of that polynomial at the point at infinity along $u$~\cite{dvir2013extensions}.

\begin{lem}[Decoding from evaluations on rich lines]\label{lem:deRich}
Let $L=\{a+\lambda u| \lambda \in \Z/p^k\Z\}\subset (\Z/p^k\Z)^n$ with $a\in (\Z/p^k\Z)^n,u\in \P (\Z/p^k\Z)^{n-1}$, $u'\in (\Z/p^\ell\Z)^n$ be such that {\em $u'\text{ (mod }p^k\text{)}=u$} and $\pi:L\rightarrow \Z_{\ge 0}$ be a function which satisfies $\sum_{x\in L} \pi(x) \ge p^\ell$.

Then, there exist elements $c_{\lambda,\balpha} \in \Q(\zeta)[z]$ (depending on $\pi,L$ and $u'$) for $\lambda \in \Z/p^k\Z$ and $\balpha\in  \Z_{\ge 0}^n$ with $\text{wt}(\balpha)< \pi(a+\lambda u)$ such that the following holds for all polynomials $f\in \Z[x_1,\hdots,x_n]$,

{\em$$ \psi_{p^k}\left(\sum\limits_{\lambda =0}^{p^k-1}\sum\limits_{\text{wt}(\balpha) < \pi(a+\lambda u) }c_{\lambda,\balpha} f^{(\balpha)}(\zeta^{a+\lambda u})\right)= f(z^{u'})\in \overline{T}_\ell .$$}
\end{lem}
For the application of $\psi_{p^k}$ in the statement of the lemma to make sense we must have its input be an element in $T_\ell^k$. In other words, we need the input of $\psi_{p^k}$ to be a polynomial in $z$ with coefficients in $\Z(\zeta)$. This is indeed the case and will be part of the proof of the lemma. We need two simple facts for the proof.

\begin{lem}[Hasse Derivatives of composition of two functions]\label{fact:com}
Let $\F$ be a field, $n\in \N$. Given a tuple of polynomials $C(y)=(C_1(y),C_2(y),\hdots,C_n(y))\in (\F[y])^n$, $w\in \N$ and $\gamma\in \F$ there exists a set of coefficients $b_{w,\balpha}\in \F$ (which depend on $C$ and $\gamma$) where $\balpha \in \Z^n_{\ge 0}$ such that for any $f\in \F[x_1,\hdots,x_n]$ we have,
$$h^{(w)}(\gamma)=\sum\limits_{\text{wt}(\balpha)\le w} b_{w,\balpha} f^{(\balpha)}(C_1(\gamma),\hdots,C_n(\gamma)), $$
where $h(y)=f(C_1(y),\hdots,C_n(y))$. 

\end{lem}

This fact follows easily from the definition of the Hasse derivative. A proof can also be found in Proposition 6 of \cite{dvir2013extensions}.\footnote{The idea is that $C_i(\gamma+z)$ can be expanded in terms of its Hasse derivatives. We use that in $f(C_1(\gamma+z),\hdots,C_n(\gamma+z))$ and expand it again using the Hasse derivative expansion of $f$ at $(C_1(\gamma),\hdots,C_n(\gamma))$. The coefficient of $z^w$ is the Hasse derivative $h^{(w)}(\gamma)$ and it will only get contributions from $f^{(\balpha)}$ with $\text{wt}(\balpha)\le w$.} We also need another fact about the isomorphism between polynomials and the evaluations of their derivatives at a sufficiently large set of points.

\begin{lem}[Computing polynomial coefficients from polynomial evaluations]\label{fact:eval}
Let $\F$ be a field and $n\in \N$. Given distinct $a_i\in \F$ and $m_i\in \Z_{\ge 0}$, let $h(y)=\prod_{i=1}^n (y-a_i)^{m_i}\in \F[y]$. We have an isomorphism between
$$\frac{\F[z]}{\left\langle h(z)\right\rangle } \longleftrightarrow \F^{\sum\limits_{i=1}^n m_i},$$
which maps every polynomial $f\in \F[z]/\langle h(z)\rangle$ to the evaluations $(f^{(j_i)}(a_i))_{i,j_i}$ where $i\in \{1,\hdots,n\}$ and $j_i\in \{0,\hdots,m_i-1\}$. 
\end{lem}

This is a simple generalization of Lemma~\ref{fact:Hasse}. It can be proven in several ways, for example it can be proven using Lemma~\ref{fact:Hasse} and the Chinese remainder theorem for the ring $\F[y]$. To prove Lemma~\ref{lem:deRich} we will need the following corollary of the fact above.

\begin{cor}[Computing a polynomial from its evaluations]\label{cor:eval}
Let $\F$ be a field and $n\in \N$. Given distinct $a_i\in \F$ and $m_i\in \Z_{\ge 0}$, let $h(y)=\prod_{i=1}^n (y-a_i)^{m_i}\in \F[y]$. Then there exists coefficients $t_{i,j}\in \F[z]$ (depending on $h$) for $i\in \{1,\hdots,n\},j\in \{0,\hdots,m_i-1\}$ such that for any $f(y)\in \F[y]$ we have,
$$\sum\limits_{i=1}^n \sum\limits_{j=0}^{m_i-1} t_{i,j} f^{(j)}(a_i)=f(z)  \text{ {\em (mod }} h(z)\text{{\em )}} \in  \F[z]/\langle h(z)\rangle.$$
\end{cor}
\begin{proof}
Lemma~\ref{fact:eval} implies that there exists a $\F$-linear map which can compute the coefficients of $1,z,z^2,\hdots, z^{\sum_{i=1}^n m_i-1}$ of $f(z) \in \F[z]/\langle h(z)\rangle$ from the evaluations $f^{(j_i)}(a_i)$ for $i\in \{1,\hdots,n\}$ and $j_i\in~\{0,\hdots,m_i-1\}$. Multiplying these coefficients with $1,\hdots, z^{\sum_{i=1}^n m_i-1}$ computes $f(z)$ in $\F[z]/\langle h(z)\rangle$.
\end{proof}

We are now ready to prove Lemma~\ref{lem:deRich}.

\begin{proof}[Proof of Lemma~\ref{lem:deRich}]
As the statement we are trying to prove is linear over $\Z$ we see that it suffices to prove the lemma in the case of when $f$ equals a monomial. Given $v\in \Z_{\ge 0}^n$ we let $m_v(x)=m_v(x_1,\hdots,x_n)=x_1^{v_1}\hdots x_n^{v_n}$ be a monomial in $\F[x_1,\hdots,x_n]$ where $\F$ is an arbitrary field (we will be working with $\F=\Q(\zeta)$ and $\F=\F_p$).
Let $C(y)=y^{u'}=(y^{u'_1},y^{u'_2},\hdots,y^{u'_n})\in (\F[y])^n$ where $u'=(u'_1,\hdots,u'_n)\in (\Z/p^\ell\Z)^n,u=(u_1,\hdots,u_n)\in \P (\Z/p^k\Z)^{n-1}$ and $u' \text{ (mod }p^k\text{)}=u$. For this proof we use the elements in $ \{0,\hdots, p^\ell -1\}$ to represent the set $\Z/p^\ell\Z$ (similarly for $(\Z/p^k\Z)^n$).

We first prove the following claim.
\begin{claim}\label{cl:inter}
Let $w\in \N, \lambda \in \Z/p^k\Z$. There exists coefficients $b'_{w,\balpha}(\lambda)\in \Q(\zeta)$ (depending on $w,\lambda$ and $C$) for $\balpha\in \Z_{\ge 0}^n$ with $\text{wt}(\balpha)\le~w$ such that for all monomials $m_v(x)\in \Z[x_1,\hdots,x_n],v\in\Z_{\ge 0}^n$ we have,
$$\sum\limits_{\text{wt}(\balpha)\le w} b'_{w,\balpha}(\lambda) m_v^{(\balpha)}(\zeta^{a+\lambda u})=\zeta^{\langle a,v\rangle} (m_v\circ C)^{(w)}(\zeta^\lambda).$$
\end{claim}
\begin{proof}
For every $\lambda \in \Z/p^k\Z$ and $w\in \N$, using Lemma~\ref{fact:com} we can find coefficients $b_{w,\balpha}(\lambda)\in \Q(\zeta)$ such that,
\begin{align}\label{eq:Der1}
(f\circ C)^{(w)}(\zeta^{\lambda})= \sum\limits_{\text{wt}(\balpha)\le w} b_{w,\balpha}(\lambda) f^{(\balpha)}(C(\zeta^{\lambda}))= \sum\limits_{\text{wt}(\balpha)\le w} b_{w,\balpha}(\lambda) f^{(\balpha)}(\zeta^{\lambda u'}),
\end{align}
for every polynomial $f\in \Q(\zeta)[x_1,\hdots,x_n]$ where $(f\circ C)(y)=f(y^{u'})\in \Q(\zeta)[y]$. As $u' \text{ (mod }p^k\text{)}=u$ and $\zeta$ is a primitive $p^k$'th root of unity in $\C$ we note that 
$$\zeta^{a+\lambda u}=\zeta^{a+\lambda u'}.$$
We now make the simple observation that for any $\balpha \in \Z_{\ge 0}^n$ and $v\in \Z_{\ge 0}^n$ we have 
$$m^{(\balpha)}_v(x)=\prod\limits_{i=1}^n\binom{v_i}{\alpha_i} x_i^{v_i-\alpha_i},$$
which implies
$$m^{(\balpha)}_v(\zeta^{a+\lambda u})=\zeta^{\langle a,v\rangle}\zeta^{-\langle \balpha,a\rangle}m^{(\balpha)}_v(\zeta^{\lambda u})=\zeta^{\langle a,v\rangle}\zeta^{-\langle \balpha,a\rangle}m^{(\balpha)}_v(\zeta^{\lambda u'}).$$
The above equation combined with \eqref{eq:Der1} for $f=m_v$ implies,
\begin{align*}
    \zeta^{\langle a,v\rangle} (m_v\circ C)^{(w)}(\zeta^\lambda)=\sum\limits_{\text{wt}(\balpha)\le w} \zeta^{\langle \balpha,a\rangle}b_{w,\balpha}(\lambda) m_v^{(\balpha)}(\zeta^{a+\lambda u}),
\end{align*}
for all $w\in \N$ and $v\in \Z_{\ge 0}^n$. Setting $b'_{w,\balpha}(\lambda)=\zeta^{\langle \balpha,a\rangle}b_{w,\balpha}(\lambda)$ we are done .
\end{proof}

Without loss of generality let us assume $\sum_{x\in L} \pi(x)=p^\ell$ (if it is greater we can reduce each of the $\pi(x)$ until we reach equality - this would just mean that our computation was done by ignoring some higher order derivatives at some of the points).

Let $h(y)\in \Z(\zeta)[y]\subseteq \Q(\zeta)[y]$ be the polynomial,
$$h(y)= \prod\limits_{\lambda \in \Z/p^k\Z}(y-\zeta^{\lambda})^{\pi(a+\lambda u)}.$$

Using Corollary~\ref{cor:eval} there is a $\Q(\zeta)[z]$-linear combination of the evaluations  $(m_v\circ C)^{(w)}(\zeta^\lambda)$ for $\lambda \in \Z/p^k\Z$ and $w<\pi(a+\lambda u)$ which can compute the element 
$$(m_v\circ C)(z)=z^{\langle v,u'\rangle} \in \Q(\zeta)[z]/\langle h(z)\rangle.$$ 
This statement along with Claim~\ref{cl:inter} leads to the following: there exist elements $c_{\lambda,\balpha} \in \Q(\zeta)[z]$ (depending on $\pi,L$ and $u'$) for $\lambda \in \Z/p^k\Z$ and $\balpha\in  \Z_{\ge 0}^n$ with $\text{wt}(\balpha)< \pi(a+\lambda u)$ such that the following holds for all monomials $m_v\in \Z[x_1,\hdots,x_n],v\in \Z_{\ge 0}$ we have,

$$\sum\limits_{\lambda =0}^{p^k-1}\sum\limits_{\text{wt}(\balpha) < \pi(a+\lambda u) }c_{\lambda,\balpha} m_v^{(\balpha)}(\zeta^{a+\lambda u})=\zeta^{\langle a,v\rangle} (m_v\circ C)(z)=\zeta^{\langle a,v\rangle}z^{\langle v,u'\rangle} \in \Q(\zeta)/\langle h(z)\rangle.$$

We claim that these are coefficients we wanted to construct in the statement of this lemma. 

All we need to show now is that $\psi_{p^k}$ is a ring homomorphism from the ring $\Z(\zeta)/\langle h(z)\rangle$ to the ring $\overline{T}_\ell=\F_p(\zeta)/\langle z^\ell -1\rangle$ and maps $\zeta^{\langle a,v\rangle}z^{\langle v,u'\rangle}$ to $z^{\langle v,u'\rangle}$. This follows from Corollary~\ref{cor:psiE} and noting 
$$\psi_{p^k}(h(z))=(z-1)^{\sum\limits_{\lambda \in\Z/p^k\Z} \pi(a+\lambda u)}=(z-1)^{p^\ell}=z^{p^\ell}-1\in \F_p[z].$$

\end{proof}

We will use the lemma above in the form of the following corollary.

\begin{cor}[Decoding $M_{p^\ell,n}$ from $U^{(\balpha)}_{p^\ell}$]\label{lem:decoding}
Let $L=\{a+\lambda u| \lambda \in \Z/p^k\Z\}\subset (\Z/p^k\Z)^n$ with $a\in (\Z/p^k\Z)^n,u\in \P (\Z/p^k\Z)^{n-1}$, $u'\in (\Z/p^\ell\Z)^n$ be such that {\em $u'\text{ (mod }p^k\text{)}=u$} and $\pi:L\rightarrow \Z_{\ge 0}$ be a function which satisfies $\sum_{x\in L} \pi(x) \ge p^\ell$ then there exists a $\Q(\zeta)[z]$-linear combination (with coefficients depending on $\pi,L$ and $u'$) of the vectors $U_{p^\ell}^{(\balpha)}(\zeta^x)$ for $x\in L$ and $\balpha$ of weight strictly less than $\pi(x)$ which under the map $\psi_{p^k}$ gives us the vector $M_{p^\ell,n}(u')$.
\end{cor}
\begin{proof}
Lemma~\ref{lem:deRich} gives us the required linear combination.
\end{proof}

The coefficients defined by the above corollary correspond to the rows of the matrix $K_S$ in the proof overview (Section~\ref{sec:proofOver}).

\section{Kakeya Set bounds over $\Z/p^k\Z$}\label{sec:pkBd}

We first prove two helper lemmas. We recall that $\GL_n(\Z/p^k\Z)$ is the set of linear isomorphisms over $(\Z/p^k\Z)^n$. They are represented by matrices whose determinants are units in $\Z/p^k\Z$.

The first lemma implies that given a large subset $D\subseteq \P (\Z/p^k\Z)^{n-1}$ there exists a $W$ in $\GL_n(\Z/p^k\Z)$ such that $M_{p^\ell,n}(W\cdot D)$ has high rank where $$W \cdot D=\{W\cdot u|u\in D\}$$ 
that is the set of directions obtained by rotating the elements in $D$ by $W$. In general, we also prove that there exists a $W$ such that $M_{p^\ell,n}(D'_W)$ has high rank where 
$$D'_W=\{u\in (\Z/p^\ell\Z)^{n}| u\text{ (mod }p^k\text{)}\in W \cdot D\}$$ 
that is the set of directions in  $(\Z/p^\ell\Z)^{n}$ which under the mod $p^k$ map give an element in $W\cdot D$. We prove the statement using a random rotation argument.

\begin{lem}\label{fm:rdRo}
Let $k,n\in \N,\epsilon \ge 0$ and $p$ be a prime. Given a set $D\subseteq \P (\Z/p^k\Z)^{n-1}$  containing at least an $\epsilon$ fraction of elements in $\P (\Z/p^k\Z)^{n-1}$ then there exists a matrix {\em $W\in \GL_n(\Z/p^k\Z)$} such that,
{\em $$\text{rank}_{\F_p} M_{p^\ell,n}(D'_W)\ge \epsilon \cdot \binom{p^\ell \ell^{-1}+n}{n}.$$}
\end{lem}
\begin{proof}
From Lemma~\ref{lem:resRk} we know that the $\F_p$ rank of $M_{p^\ell,n}((\P \Z/p^\ell\Z)^{n-1})$ is at least $\binom{p^\ell l^{-1}+n}{n}$. This means there exists a set of rows $V$ of size at least $\binom{p^\ell l^{-1}+n}{n}$ in $\Coeff(M_{p^\ell,n}((\P \Z/p^\ell\Z)^{n-1}))$ which are $\F_p$-linearly independent.

We pick $W\in \GL_n(\Z/p^k\Z)$ uniformly at random and as in the statement of the lemma consider the set $D'_W=\{u\in (\Z/p^\ell\Z)^{n}| u\text{ (mod }p^k\text{)}\in W\cdot D\}$. The key claim about $D'_W$ is the following.

\begin{claim}
Any given row $v\in V$ will appear in {\em$\Coeff(M_{p^\ell,n}(D'_W))$} with probability at least $\epsilon$.  
\end{claim}
\begin{proof}
As $\GL_n(\Z/p^k\Z)$ acts transitively on the set $\P (\Z/p^k\Z)^{n-1}$ and $|D|\ge \epsilon |\P (\Z/p^k\Z)^{n-1}|$ we see that $D'_W$ will contain a particular $u\in \P (\Z/p^\ell\Z)^{n-1}$ with probability at least $\epsilon$. The row $v\in V$ will be within $\Coeff(M_{p^\ell,n}(u))$ for some $u\in \P (\Z/p^\ell\Z)^{n-1}$. This means that $v\in V$ will appear in $\Coeff(M_{p^\ell,n}(D'_W))$ with probability at least $\epsilon$. 
\end{proof}

Now by linearity of expectation we see that the expected number of rows $V$ appearing in $\Coeff(M_{p^\ell,n}(D'_W))$ is at least $\epsilon|V|$. This means there exists a choice of $W$ such that an $\epsilon$ fraction of elements in $V$ do appear in $\Coeff(M_{p^\ell,n}(D'_W))$. As all these rows are linearly independent then by Lemma~\ref{lem:coeffR} we see that the $\F_p$-rank of $M_{p^\ell,n}(D'_W)$ must be at least $\epsilon|V|\ge \epsilon \binom{p^\ell l^{-1}+n}{n}$.
\end{proof}

The second lemma lower bounds the size of $(m,\epsilon)$-Kakeya sets by the rank of a sub-matrix of $M_{p^\ell, n}$ with rows corresponding to directions with rich lines.

\begin{lem}\label{fm:rkBd}
Let $k,n\in \N$ and $p$ be a prime. Let $S\subseteq (\Z/p^k\Z)^n$ be a $(m,\epsilon)$-Kakeya set and $D\subseteq \P (\Z/p^k\Z)^{n-1},|D|\ge \epsilon |\P (\Z/p^k\Z)^{n-1}|$ such that for every $u\in D$ we have a $m$-rich line $L_u$ with respect to $S$ in direction $u$. Then we have the following bound,
{\em $$ |S|\binom{\lceil p^\ell m^{-1} \rceil+n-1}{n} \ge \text{rank}_{\F_p} M_{p^\ell,n}(D'),$$}
for all $\ell \ge \log_p(m)$ where {\em$D'=D'_I=\{u\in (\Z/p^\ell\Z)^{n}| u\text{ (mod }p^k\text{)}\in D\}.$}
\end{lem}
\begin{proof}
We set
$$b=\lceil p^\ell m^{-1} \rceil.$$
Consider the family of row vectors, $U_{p^\ell}^{(\balpha)}(\zeta^x)$ where $x\in S$ and $\text{wt}(\balpha)<b$. For a $u\in D$, let $L_u$ be a $m$-rich line with respect to $S$ in direction $u$. Let $\cU_{L_u,p^\ell}$ be the matrix constructed by concatenating the row vectors $U_{p^\ell}^{(\balpha)}(\zeta^x)$ where $x\in L_u$ and $\text{wt}(\balpha)<b$. By construction as the rows in $\{\cU_{L_d,p^\ell}\}_{u\in D}$ correspond to tuples $x\in S$ and $\balpha\in \Z_{\ge 0}^n, \text{wt}(\balpha)<b$ we have the following bound,

\begin{align}\label{eq:NBd}
    |S|\binom{b+n-1}{n} \ge \text{crank}_{\Q(\zeta)} \{\cU_{L_u,p^\ell}\}_{u\in D}.
\end{align}

Let $\pi_u:L_u\rightarrow \Z_{\ge 0}$ be a weight function which gives weight $b$ to points in $S\cap L_u$ and $0$ elsewhere. Using Corollary~\ref{lem:decoding} for the line $L_u$ and function $\pi_u$, for all $u' \in (\Z/p^\ell\Z)^{n}$ such that $u' \text{ (mod }p^k\text{)}=u$, we can now construct row vectors $C_{u'}$ such that $\psi_{p^k}(C_{u'}\cU_{L_u,p^\ell})=M_{p^\ell,n}(u')$.  For convenience, for every $u\in \P (\Z/p^k\Z)^{n-1}$ we define $\cC_u$ and $\cM_u$ as the matrices whose row vectors are $C_{u'}$ and $M_{p^\ell,n}(u')$ respectively for all $u'\in (\Z/p^\ell\Z)^{n}$ such that $u' \text{ (mod }p^k\text{)}=u$. We note, $\psi_{p^k}(\cC_{u}\cU_{L_u,p^\ell})=\cM_{p^\ell,n}(u)$.

We use Lemma~\ref{lem:crankMatMult} on $\{\cU_{L_u,p^\ell}\}_{u\in D}$ being multiplied by $\cC_u,u\in D$ and applying Lemma~\ref{lem:quoRank} next to get,
\begin{align*}
    \text{crank}_{\Q(\zeta)} \{\cU_{L_u,p^\ell}\}_{u\in D}&\ge \text{crank}_{\Q(\zeta)} \{\cC_{u}\cU_{L_u,p^\ell}\}_{u\in D}\\
    &\ge  \text{crank}_{\F_p} \{\cM_{p^\ell,n}(u)\}_{u\in D}=\text{rank}_{\F_p} M_{p^\ell,n}(D'),
\end{align*}
where we recall $D'=D'_I=\{u\in (\Z/p^\ell\Z)^{n}| u\text{ (mod }p^k\text{)}\in D\}.$
Using the above inequality with \eqref{eq:NBd} we have,
\begin{align*}
|S|\binom{b+n-1}{n}=    |S|\binom{\lceil p^\ell m^{-1} \rceil+n-1}{n}\ge  \text{rank}_{\F_p} M_{p^\ell,n}(D').
\end{align*}
\end{proof}

We are now ready to prove Theorem~\ref{thm:Smain}.

\begin{repthm}{thm:Smain}
Let $k,n,p\in \N$ with $p$ prime. Any $(m,\epsilon)$-Kakeya set $S\subseteq (\Z/p^k\Z)^n$ satisfies the following bound,
$$|S| \ge \epsilon\cdot\left(\frac{m^n}{(2(k+\lceil\log_p(n)\rceil))^n}\right).$$
When $p>n$, we also get the following stronger bound for $(m,\epsilon)$-Kakeya set $S$ in $(\Z/p^k\Z)^n$,
$$|S|\ge \epsilon\cdot\left(\frac{m^n}{(k+1)^n}\right) (1+n/p)^{-n}.$$
\end{repthm}
\begin{proof}
Set $\ell = k+\lceil \log_p(n)\rceil$. Let $D$ be the set of directions in $\P (\Z/p^k\Z)^{n-1}$ which have $m$-rich lines. We know $D$ contains at least an $\epsilon$ fraction of directions. By Lemma~\ref{fm:rdRo} there exists a matrix $W\in \GL_n(\Z/p^k\Z)$ such that,

$$\text{rank}_{\F_p} M_{p^\ell,n}(D'_W)\ge \epsilon \cdot \binom{p^\ell \ell^{-1}+n}{n},$$
where $D'_W=\{u\in (\Z/p^\ell\Z)^{n}| u\text{ (mod }p^k\text{)}\in W \cdot D\}.$

We now note $W\cdot S$ is also an $(m,\epsilon)$-Kakeya set with $W\cdot D$ as the set of directions with $m$-rich lines. We now apply Lemma~\ref{fm:rkBd} with the inequality above to get,

$$|S|\binom{p^\ell m^{-1}+n}{n}\ge |S|\binom{\lceil p^\ell m^{-1}\rceil+n-1}{n}\ge \text{rank}_{\F_p} M_{p^\ell,n}(D'_W) \ge \epsilon \cdot \binom{p^\ell \ell^{-1} +n}{n}.$$

For convenience we set $a=\lceil \log_p(n)\rceil$ which implies $p^{a}\ge n\ge p^{a-1}$. The above inequality implies,
\begin{align*}
    |S| &\ge \epsilon \cdot \prod\limits_{i=1}^n \frac{p^k p^{a} (k+a)^{-1} +i}{p^k p^{a} m^{-1}+i}\\
    &\ge \epsilon \cdot \prod\limits_{i=1}^n \frac{ m(k+a)^{-1} +i m p^{-k} p^{-a}}{1+i m p^{-k} p^{-a}}\\
    &\ge \epsilon \cdot \frac{m^n}{(k+a)^{n}} \prod\limits_{i=1}^n\left(1+i p^{-a} \right)^{-1}.
\end{align*}
For $i\le n$ we have $\left(1+i p^{-a} \right)^{-1}\le 2$ which completes the proof. Note, the second half of the theorem follows from observing that for $p>n$ we have $\left(1+i p^{-a} \right)^{-1}\le 1+n/p$ and $a=\lceil\log_p(n)\rceil=1$.
\end{proof}

\section{Kakeya Set bounds over $\Z/N\Z$ for general $N$}\label{sec:Nbd}

We need some simple facts about $\Z/p^kN\Z$ for $p$ prime and $p$ and $N$ co-prime which follow from the Chinese remainder Theorem.

\begin{lem}[Geometry of $\Z/p^kN\Z$]\label{fact:remainder}
Let $p,N,n,k\in \N,R=\Z/p^kN\Z,R_1=\Z/N\Z,R_0=\Z/p^k\Z$ with $p$ prime and co-prime to $N$.  Using the Chinese remainder theorem we know that $R^n$ is isomorphic to $R_0^n\times R_1^n$ and $\P R^{n-1}$ is isomorphic to $\P R_0^{n-1}\times \P R_1^{n-1}$.  Finally, any line $L=\{a+\lambda u| \lambda \in R\}$ with direction $u=(u_0,u_1)\in R_0^{n} \times R_1^{n}$ in $R^n$ is equivalent to the product of a line $L_0\subset R_0^n$ in direction $u_0$ and a line $L_1\subset R_1^n$ in direction $u_1$. 
\end{lem}

We briefly discuss the proof strategy first. We use ideas from \cite{dhar2021proof} to prove Theorem~\ref{main}. We give the idea for $N=p^{k_1}q^{k_2}$ where $p$ and $q$ are distinct primes. Let $S$ be Kakeya set which contains lines $L_u$ in the direction $u\in \P(\Z/N\Z)^{n-1}$. Using Lemma~\ref{fact:remainder} we know that the line $L_u$, $u=(u_p,u_q)\in \P (\Z/p^{k_1}\Z)^{n-1}\times \P (\Z/q^{k_2}\Z)^{n-1}$ can be decomposed as a product of lines $L_p(u_p,u_q)$ over $\Z/p^{k_1}\Z$ and $L_q(u_p,u_q)$ over $\Z/q^{k_2}\Z$  with directions $u_p\in \P (\Z/p^{k_1}\Z)^{n-1}$ and $u_q\in\P (\Z/q^{k_2}\Z)^{n-1}$ respectively. Note that $L_p(u_p,u_q)$ and $L_q(u_p,u_q)$ depend on $u=(u_p,u_q)$, and not just on $u_p$ or $u_q$ respectively. 

Let $\zeta$ be a primitive $p^{k_1}$'th root of unity in $\C$ and $\Ind_y$ be the indicator vector of the point $y\in (\Z/q^{k_2}\Z)^n$. We then examine the span of vectors 
$$U_{p^{k_1}}(\zeta^x)\otimes \Ind_y=U^{((0,\hdots,0))}_{p^{k_1}}(x)\otimes \Ind_y$$
for $x\in L_p(u_p,u_q),y\in L_q(u_p,u_q),(u_p,u_q)\in \P (\Z/p^{k_1}\Z)^{n-1}\times \P (\Z/q^{k_2}\Z)^{n-1}$. As there is one vector for each point in $S$, the dimension of the space spanned by these vectors is at most $|S|$. We then use the decoding procedure for $U_{p^{k_1}}$ from Corollary~\ref{lem:decoding} to linearly generate vectors 
$$M_{p^{k_1},n}(u_p)\otimes \Ind_y$$
for $y\in L_q(u_p,u_q),(u_p,u_q)\in \P (\Z/p^{k_1}\Z)^{n-1}\times \P (\Z/q^{k_2}\Z)^{n-1}$ from the vectors $U_{p^{k_1}}(\zeta^x)\otimes \Ind_y$ for $x\in L_p(u_p,u_q),y\in L_q(u_p,u_q)$. This means the dimension of the span of $M_{p^{k_1},n}(u_p)\otimes y$ for $y\in L_q(u_p,u_q),(u_p,u_q)\in \P (\Z/p^{k_1}\Z)^{n-1}\times \P (\Z/q^{k_2}\Z)^{n-1}$ lower bounds the size of $S$. Note that $M_{p^{k_1},n}(u_p)$ only depends on $u_p$ and not on $u_q$.

We pick the largest subset of rows $V$ in $\Coeff(M_{p^{k_1},n}(\P (\Z/p^{k_1}\Z)^{n-1})$ which are linearly independent. By Lemma~\ref{lem:1rank}, $V$ has large size. Any vector $v\in V$ will correspond to a row in $\Coeff(M_{p^{k_1},n}(u_p))$ for some $u_p\in \P (\Z/p^k\Z)^{n-1}$. For that $u_p$, the span of $\Ind_y$ for $y\in L_q(u_p,u_q)$ as $u_q$ varies in $\P (\Z/q^{k_2}\Z)^{n-1} $ will have dimension exactly equal to the size of the set 
$$S_v=\bigcup\limits_{u_q\in \P (\Z/q^{k_2}\Z)^{n-1} } L_q(u_p,u_q)$$ 
which is a Kakeya set in $(\Z/q^{k_2}\Z)^n$ and has large size by Theorem~\ref{thm:Smain}. Simple properties of the tensor product (Lemma~\ref{lem:crankTensorProduct}) now imply that $v\otimes \Ind_y$ for $v\in V$ and $y\in S_v$ are linearly independent which gives us a rank bound and hence a lower bound on the size of $S$

Using $U^{(\balpha)}_{p^\ell}$ for $\balpha\in \Z_{\ge 0}^n,\text{wt}(\balpha)<p^{\ell-k}$ as in the proof of Theorem~\ref{thm:Smain} gives us better constants. The proof for an arbitrary number of distinct prime factors applies the above argument inductively.

\begin{repthm}{main}
Let $n\in \N$ and $R=\Z/N\Z$  where $N=p_1^{k_1}\hdots p_r^{k_r}$ with distinct primes $p_1,\hdots,p_r$ and $k_1,\hdots,k_r\in \N$. Any Kakeya set $S$ in $R^n$ satisfies,
$$|S|\ge N^n\prod\limits_{i=1}^r (2(k_i+\lceil \log_{p_i}(n)\rceil))^{-n}.$$

When $p_1,\hdots,p_r\ge n$, we also get the following stronger lower bound for the size of a Kakeya set $S$ in $(\Z/N\Z)^n,N=p_1^{k_1}\hdots p_r^{k_r}$,
$$|S|\ge N^n\prod\limits_{i=1}^r(k_i+1)^{-n}(1+n/p_i)^{-n}.$$
\end{repthm}
\begin{proof}
Our proof will apply induction on $r$. The case of $r=1$ is Theorem~\ref{thm:Smain}. Let the theorem be known for $r$ distinct prime factors.
Let $N=p_0^{k_0}p_1^{k_1}\hdots p_r^{k_r}$ with $N_1=p_1^{k_1}\hdots p_r^{k_r}$. Let $R=\Z/N\Z$, $R_0=\Z/p_0^{k_0}\Z$ and $R_1=\Z/N_1\Z$. 

Consider $S$ a Kakeya set in $R^n$ and a set of lines indexed by directions $u\in \P R^{n-1}$ such that $L(u)=\{a(u)+\lambda u|\lambda \in R\}\subset R^n$ is a line in the direction $u$ and is contained in $S$. By Lemma~\ref{fact:remainder} we know $u$ can be written as a tuple $(u_0,u_1) \in \P R_0^{n} \times \P R_1^{n}$ and $L(u)$ is a product of a line $L_{0}(u_0,u_1)\subset R_0^n$ in direction $u_0$ and line $L_{1}(u_0,u_1)\subset R_1^n$ in direction $u_1$. We note $L_1(u_0,u_1)$ can actually depend on $u_0$ (similarly for $L_0(u_0,u_1)$ and $u_1$). This will not be the case only when $S$ itself is a product of sets from $R_0^n$ and $R_1^n$.

Let
\begin{align}\label{eq:b2}
\ell = k_0+\lceil \log_{p_0}(n)\rceil \text{ and } b =p^{\lceil \log_{p_0}(n)\rceil}.
\end{align} 

Let $\zeta_0$ be a complex primitive $p_0^{k_0}$'th root of unity. We define $\cU_{L_{0}(u_0,u_1)}$ as the matrix constructed by concatenating the row vectors $U_{p_0^\ell}^{(\balpha)}(\zeta_0^x)$ where $x\in L_{0}(u_0,u_1)$ and $\text{wt}(\balpha)<b$.

Let $Y_{u_0,u_1}$ be a matrix whose columns are labeled by points in $R_1^n$ and rows labeled by points in $L_1(u_0,u_1)$ such that its $y$'th row is $\Ind_y$ where $\Ind_y$ is the indicator vector of $y\in R_1^n$. We will work with the family of matrices $\cU_{L_{0}(u_0,u_1)}\otimes Y_{u_0,u_1}$ for $(u_0,u_1)\in \P R_0^{n-1}\times \P R_1^{n-1}$.

The following claim connects this family to the size of $S$.

\begin{claim}\label{claim:set1}
{\em $$\crank_{\Q(\zeta_0)} \{\cU_{L_{0}(u_0,u_1)}\otimes Y_{u_0,u_1}\}_{u_0\in \P R_0^{n-1},u_1\in \P R_1^{n-1}} \le |S|\binom{b+n-1}{n}.$$}
\end{claim}
\begin{proof}
As we are working with matrices whose entries are in a field, row and column ranks of matrices are identical. The rows of $\cU_{L_{0}(u_0,u_1)}\otimes Y_{u_0,u_1}$ are the vectors $U_{p_0^\ell}^{(\balpha)}(\zeta_0^x)\otimes \Ind_y$ for $x\in~L_0(u_0,u_1)$, $\text{wt}(\balpha)<b$ and $y\in L_1(u_0,u_1)$. Hence there are at most $|S|\binom{b+n-1}{n}$ distinct rows in the set of $\cU_{L_{0}(u_0,u_1)}\otimes Y_{u_0,u_1}$ for $u_0\in \P R_0^{n-1},u_1\in \P R_1^{n-1}$.
\end{proof}

We now use Lemmas \ref{lem:crankMatMult} and \ref{lem:quoRank} and Corollary~\ref{lem:decoding} to prove the following claim. For convenience we let $R_0(\ell)=\Z/p_0^\ell\Z$ and let 
$$J=\{(u',u_0)\in \P R_0(\ell)^{n-1}\times \P R_0^{n-1}| u' \text{ (mod }p_0^{k_0}\text{)}=u_0\}.$$

\begin{claim}\label{claim:set2}
{\em $$\crank_{\F_{p_0}} \{M_{p_0^\ell,n}(u')\otimes Y_{u_0,u_1}\}_{(u',u_0,u_1)\in J\times \P R_1^{n-1}}\le \crank_{\Q(\zeta_0)} \{\cU_{L_{0}(u_0,u_1)}\otimes Y_{u_0,u_1}\}_{u_0\in \P R_0^{n-1},u_1\in \P R_1^{n-1}}.$$}
\end{claim}
\begin{proof}
Let $\pi_{u_0,u_1}:L_0(u_0,u_1)\rightarrow \Z_{\ge 0}$ for $u_0\in \P R_0^{n-1},u_1\in \P R_1^{n-1}$ be a family of functions which takes the constant value $b$ everywhere. Using Corollary~\ref{lem:decoding} for the line $L_0(u_0,u_1)$ and function $\pi_{u_0,u_1}$, for any $u' \in \P R_0(\ell)^{n-1}$ such that $(u',u_0)\in J$ we can construct row vectors $C_{u'}$ such that 
$$\psi_{p_0^{k_0}}(C_{u'}\cdot \cU_{L_{0}(u_0,u_1)})=M_{p_0^\ell,n}(u').$$ 

Now, Lemma~\ref{fact:multOfKronecker} implies,

\begin{align*}\label{eq:long}
\psi_{p_0^{k_0}}((C_{u'}\otimes I_{L_1(u_0,u_1)})\cdot (\cU_{L_{0}(u_0,u_1)}\otimes Y_{u_0,u_1}))&= \psi_{p_0^{k_0}}(C_{u'}\cdot \cU_{L_{0}(u_0,u_1)})\otimes Y_{u_0,u_1}\\
&= M_{p_0^{\ell},n}(u')\otimes Y_{u_0,u_1} \numberthis,
\end{align*}
where $I_{L_1(u_0,u_1)}$ is the identity matrix of size $|L_1(u_0,u_1)|=N_1$.

Applying Lemma~\ref{lem:crankMatMult} implies that
$$\crank_{\Q(\zeta_0)} \{\cU_{L_{0}(u_0,u_1)}\otimes Y_{u_0,u_1}\}_{(u_0,u_1)\in \P R_0^{n-1}\times \P R_1^{n-1}}$$
is larger than 
$$\crank_{\Q(\zeta_0)} \{(C_{u'}\otimes I_{L_1(u_0,u_1)})\cdot (\cU_{L_{0}(u_0,u_1)}\otimes Y_{u_0,u_1})\}_{(u',u_0,u_1)\in J\times \P R_1^{n-1}}.$$

Applying Lemma~\ref{lem:quoRank} using \eqref{eq:long} implies that 

$$\crank_{\Q(\zeta_0)} \{(C_{u'}\cdot \cU_{L_{0}(u_0,u_1)})\otimes Y_{u_0,u_1}\}_{(u',u_0,u_1)\in J\times \P R_1^{n-1}} $$
is greater than
$$\crank_{\F_{p_0}} \{M_{p_0^{\ell},n}(u')\otimes Y_{u_0,u_1}\}_{(u',u_0,u_1)\in J\times \P R_1^{n-1}}. $$

Combining these two inequalities using Lemma~\ref{fact:multOfKronecker} completes the proof.
\end{proof}

We want to apply Lemma~\ref{lem:crankTensorProduct} on $\crank_{\F_{p_0}} \{M_{p_0^{\ell},n}(u')\otimes Y_{u_0,u_1}\}_{(u',u_0,u_1)\in J\times \P R_1^{n-1}}$. To that end, we need the following two claims.

\begin{claim}
For a given $u_0 \in \P R_0^{n-1}$,
{\em $$\crank_{\F_{p_0}} \{Y_{u_0,u_1}\}_{u_1\in \P R_1^{n-1}} \ge N_1^n\prod\limits_{i=1}^r (2(k_i+\lceil \log_{p_i}(n)\rceil))^{-n}.$$}
\end{claim}
\begin{proof}
As we are working with a matrix with entries in $\F_{p_0}$ the row space and column space of the matrix formed by concatenating $\{Y_{u_0,u_1}\}_{u_1\in \P R_1^{n-1}}$ along the columns will be equal. The row space of $Y_{u_0,u_1}$ is the space spanned by $\Ind_y$ for $y\in L_1(u_0,u_1)$. If we consider all the rows from matrices in the set $\{Y_{u_0,u_1}\}_{d_1\in R_1^n}$ we will obtain vectors $\Ind_y$ for all $y \in \bigcup_{d_1\in  \P R_1^{n-1}} L_1(u_0,u_1)$. 
This means $\crank_{\F_{p_0}} \{Y_{u_0,u_1}\}_{u_1\in \P R_1^{n-1}}$ is exactly equal to the size of $\bigcup_{u_1\in \P R_1^{n-1}} L_1(u_0,u_1)$ but that is a Kakeya set in $(\Z/N_1\Z)^n$. Finally, applying the induction hypothesis we are done.
\end{proof}

\begin{claim}
{\em$$\crank_{\F_{p_0}} \{M_{p_0^{\ell},n}(u')\}_{u'\in \P R_0(\ell)^{n-1}} \ge \binom{p_0^\ell \ell^{-1}+n}{n}.$$}
\end{claim}
\begin{proof}
This follows from Lemmas \ref{lem:1rank} and \ref{lem:resRk}.
\end{proof}

Now applying Lemma~\ref{lem:crankTensorProduct} and the above two claims we see that
$$\crank_{\F_{p_0}} \{M_{p_0^{\ell},n}(u')\otimes Y_{u_0,u_1}\}_{(u',u_0,u_1)\in J}\ge \left(N_1^n\prod\limits_{i=1}^r (2(k_i+\lceil \log_{p_i}(n)\rceil))^{-n}\right)\binom{p_0^\ell \ell^{-1}+n}{n}. $$
Applying the above equation and Claims \ref{claim:set1} and \ref{claim:set2} we obtain the following bound on $|S|$.
$$|S|\binom{b+n-1}{n}\ge \left(N_1^n\prod\limits_{i=1}^r (2(k_i+\lceil \log_{p_i}(n)\rceil))^{-n}\right)\binom{p_0^\ell \ell^{-1}+n}{n}.$$
Recall, $\ell=k_0+\lceil \log_{p_0}(n)\rceil$ and $b=p_0^{\lceil \log_{p_0}(n)\rceil}\ge n$. Substituting these in the inequality above we get.
\begin{align*}
    \left(N_1^n\prod\limits_{i=1}^r (2(k_i+\lceil \log_{p_i}(n)\rceil))^{-n}\right)^{-1}|S|  &\ge  \prod\limits_{i=1}^n \frac{p_0^{k_0} p_0^{\lceil \log_{p_0}(n)\rceil} (k_0+\lceil \log_{p_0}(n)\rceil)^{-1} +i}{p_0^{\lceil \log_{p_0}(n)\rceil} +i}\\
    &\ge \prod\limits_{i=1}^n \frac{ p_0^{k_0}(k_0+\lceil \log_{p_0}(n)\rceil)^{-1} +i p_0^{-\lceil \log_{p_0}(n)\rceil}}{1+i  p_0^{-\lceil \log_{p_0}(n)\rceil}}\\
    &\ge \frac{p_0^{k_0n}}{(k_0+\lceil \log_{p_0}(n)\rceil)^{n}} \prod\limits_{i=1}^n\left(1+i p_0^{-\lceil \log_{p_0}(n)\rceil} \right)^{-1}.
\end{align*}
For $i\le n$ we have $\left(1+i p_0^{-\lceil \log_{p_0}(n)\rceil} \right)^{-1}\le 2$. This observation with the inequality above completes the proof.

Note for $p_0>n$ we have $\left(1+i p_0^{-\lceil \log_{p_0}(n)\rceil} \right)^{-1}\le 1+n/p_0$ and $\lceil\log_{p_0}(n)\rceil=1$. This with the inequality above and suitably modifying the induction hypothesis will give us the second half of this theorem.
\end{proof}

\section{Constructing small Kakeya sets}\label{sec:const}
\begin{repthm}{thm:con}
Let $s,n\in \N$, $p$ be a prime and $k=(p^{s+1}-1)/(p-1)$. There exists a Kakeya set $S$ in $(\Z/p^k\Z)^n$ such that,
$$|S| \le \sum\limits_{i=1}^n \frac{p^{ki}}{k^{i-1}(1-p^{-1})^{i-1}}\le \frac{p^{kn}}{k^{n-1}(1-p^{-1})^{n}}.$$
\end{repthm}
\begin{proof}
The construction will be inductive. For $n=1$ it $S$ is the whole set.
For a general $n$ we first construct a set which has lines with directions $u=\{1,u_2,\hdots,u_n\}\in \P (\Z/p^k\Z)^{n-1}$. We will need the following claim.
\begin{claim}
There exists a function $g: \Z/p^k\Z\rightarrow \Z/p^k\Z$ such that for any fixed $t\in \Z/p^k\Z$ the function $x\rightarrow tx-g(x),x\in \Z/p^k\Z$ has an image of size at most $p^k/(k(1-p^{-1}))$.
\end{claim}
\begin{proof}
We write an element $u\in \Z/p^k\Z$ in its $p$-ary expansion $a_0+a_1p+\hdots a_{k-1}p^{k-1}$ where $a_i\in \{0,\hdots,p-1\}$. $g$ will be of the form,
$$g(a_0+a_1p+\hdots a_{k-1}p^{k-1})=\sum\limits_{j=0}^{k-1} a_jc_jp^j,$$
where $c_i \in \{0,\hdots, p^s-1\}$. We write the number $c_i$ in its $p$-ary form $c_i(0)+c_i(1)p+\hdots c_i(s-1)p^{s-1}, c_i(j)\in \{0,\hdots, p-1\}$ and represent $c_i$ with the tuple $(c_i(0),\hdots,c_i(s-1))$. 

We will set $c_0,c_1,\hdots,c_{p^k-1}$ iteratively. $c_0=p^s-1=(p-1,\hdots,p-1)$. Once we have set $c_i$, if $c_i(s-1)>0$ we set $c_{i+1}=(c_i(0),\hdots,c_i(s-1)-1)$. If $c_i(s-1)=0$ and $c_i\ne 0$ then let $\alpha$ be the largest index such that $c_i(\alpha)=0$ and $c_{i}(\alpha-1)\ne 0$. We then set $c_{i+1}=\hdots=c_{i+s-\alpha}=c_i$ and $c_{i+s+1-\alpha}=(c_i(0),\hdots,c_i(\alpha-1)-1,p-1,\hdots,p-1)$. If $c_i=0$ then all the later $c_j,j>i$ are also set to $0$.  We note in this sequence of $c_i$s each number in $\{1,\hdots, p^s-1\}$ only appears at most as many times as the number of trailing zeros plus 1 in its $p$-ary expansion. This means the number of zero equals,
$$k-\sum\limits_{i=0}^{s-1} (i+1)p^{s-i-1}(p-1)= \frac{p^{s+1}-1}{p-1}-\sum\limits_{i=0}^{s-1} (i+1)p^{s-i-1}(p-1)=s+1.$$

The following property of $c_i$ quickly follows by the construction.
\begin{stat}\label{stat:Ex}
Given a fixed $t'\in \{0,\hdots,p^s-1\}$, let $\beta$ be the smallest index such that $c_\beta=t'$. This construction guarantees that {\em $c_{\beta+i}-t'=0\text{(mod }p^{s-i}\text{)}$} for $i=\{0,\hdots,s-1\}$.

\end{stat}

For any fixed $t$ consider $tu-g(u)-tv+g(v)$ where $u=a_0+a_1p+\hdots a_{k-1}p^{k-1}$ and $v=b_0+b_1p+\hdots b_{k-1}p^{k-1}$,
\begin{align}\label{eq:ex}
    tu-g(u)-tv+g(v)=\sum\limits_{j=0}^{k-1} (t-c_j)(a_j-b_j)p^j.
\end{align}
Let $t'=t\text{(mod }p^s\text{)}$ and $\beta$ be the smallest index such that $c_\beta=t'$. If $u_i=v_i$ for $i< \beta$ then using Statement \ref{stat:Ex} on \eqref{eq:ex} we have,
$$\left(tu-g(u)-tv+g(v) \right)\text{(mod }p^{s+\beta}\text{)}=0.$$
This means after fixing the first $\beta$ coordinates of $u$, $tu-g(u)$ can take at most $p^{k-s-\beta}$ values. This means the function $tu-g(u)$ can take at most $p^{k-s}\le p^k/(k(1-p^{-1}))$ many values.
\end{proof}
 The set 
$$S_n=\{(t,tu_2-g(u_2),\hdots,tu_n-g(u_n))|t,u_2,\hdots,u_n\in \Z/p^k\Z\}$$
contains a line in every direction $\{1,u_2,\hdots,u_n\}\in \P (\Z/p^k\Z)^{n-1}$. By the claim above, for a fixed $t$ the function $u\rightarrow tu-g(u)$ will have an image of size at most $p^k/(k(1-p^{-1}))$. This ensures $S_n$ is of size at most $p^{kn}/(k^n(1-p^{-1})^n)$. To add points with lines in other directions $(0,u_2,\hdots,u_n)\in \P (\Z/p^k\Z)^{n-1}$ we simply need to add a Kakeya set in $(\Z/p^k\Z)^{n-1}$ which we can do using the induction hypothesis, completing the construction.
\end{proof}
\begin{comm}
Let $q$ be a prime power, $s,n\in \N$ and $k=q^s-1/(q-1)$. The construction above can be adapted to find Kakeya sets in $(\F_q[x]/\langle x^k\rangle)^n$ of size,
$$\sum\limits_{i=1}^n \frac{q^{ki}}{k^{i-1}(1-q^{-1})^{i-1}}\le \frac{q^{kn}}{k^{n-1}(1-q^{-1})^{n}}.$$
\end{comm}

%%% AUTHOR: optional appendix here

\appendix
\section{Proof of Lemma \ref{lem:1rank}}\label{ap:rk}

We slightly sharpen the analysis of \cite{arsovski2021padic} here. We first show $M_{p^k,1}$ has an explicit decomposition as a product of lower and upper triangular matrices.

\begin{lem}[Lemma 5 in \cite{arsovski2021padic}]\label{lem:diag}
Let $V_m$ for $m\in \N$ be an $m\times m$ matrix whose row and columns are labelled by elements in $\{0,\hdots,m-1\}$ such that its $i,j$th entry is $z^{ij}\in \Z[z]$. 

In this setting, there exists a lower triangular matrix $L_{m}$ over $\Z[z]$ with ones on the diagonal such that its inverse is also lower triangular with entries in $\Z[z]$ with ones on the diagonal, and an upper triangular matrix $D_{m}$ over $\Z[z]$ whose rows and columns are indexed by points in $\{0,\hdots,m-1\}$ such that the $j$th diagonal entry for $j\in \{0,\hdots,m-1\}$ equals

 \begin{align*}
 D_{m}(j,j)=      \prod_{i=0}^{j-1} (z^j-z^i)
 \end{align*}
 such that $V_{m}=L_{m}D_{m}$.
\end{lem}
This statement is precisely Lemma 5 in \cite{arsovski2021padic}. We will also need Lucas's theorem from \cite{Lucas}.

% We give an alternate proof for completeness.
% \begin{proof}
% Let $E_{i,j}(x)$ be the elementary matrix which on right multiplication with a matrix $M$ adds $x$ times the $i$th row to the $j$th row. If $i<j$ then we see that $E_{i,j}(x)$ is lower triangular and $E_{i,j}(x)^{-1}=E_{i,j}(-x)$.
% We prove the lemma by proving the following statement by induction: There exists a sequence of lower triangular elementary matrices $E_{i_1,j_1}(f_1),\hdots,E_{i_{n(m)},j_{n(m)}}(f_{n(m)})$ such that  $\prod_{k=1}^{n(m)} E_{i_k,j_k}(f_k) V_m=D_m$ where $D_m$ is upper triangular and its diagonal entries are as given in the lemma statement. It is clear that proving this statement also proves the Lemma.

% For $m=1$ the statement is trivial.
% Let it be true for some $m$. In $V_{m+1}$ we take its $0$th row and consider the matrix $D'_{m+1}=\prod_{j=1}^{m} E_{0,j}(-z^j) V_{m+1}$. This operation will lead to the first column of $D'_{m+1}$ being $(1,0,\hdots,0)$ and the sub-matrix of $D'_{m+1}$ after deleting the $0$th row and the $0$th column is going to be $(z-1)V_m$. We now apply the induction hypothesis to complete the proof.
% \end{proof}

\begin{thm}[Lucas's Theorem~\cite{Lucas}]\label{thm:luc}
Let $p$ be a prime and Given two natural numbers $a$ and $b$ with expansions $a_k p^k+\hdots+a_1p+a_0$ and  
$b_kp^k+\hdots+b_0$ in base $p$ we have,
\em{$$\binom{a}{b} \text{(mod }p\text{)}=\prod\limits_{i=0}^k\binom{a_i}{b_i} \text{(mod }p\text{)} .$$}
A particular consequence is that $\binom{a}{b} \mod p$ is non-zero if and only if every digit in base-$p$ of $b$ is at most as large as every digit in base-$p$ of $a$.
\end{thm}

\begin{replem}{lem:1rank}
The $\F_p$-rank of $M_{p^\ell,n}$ is at least 
{\em $$\text{rank}_{\F_p} M_{p^\ell,n} \ge \binom{\lceil p^\ell \ell^{-1}\rceil +n}{n}.$$}
\end{replem}

\begin{proof}
Using the previous lemma we note that $V_{p^\ell}=L_{p^\ell}D_{p^\ell}$. Under the ring map $f$ from $\Z[z]$ to $\Z[z]/\langle z^{p^\ell}-1\rangle$, $V_{p^\ell}=L_{p^\ell}D_{p^\ell}$ becomes $M_{p^\ell,1}=\overline{L}_{p^\ell}\overline{D}_{p^\ell}$ where $\overline{L}_{p^\ell}$ and $\overline{D}_{p^\ell}$ are the matrices $f(L_{p^\ell})$ and $f(D_{p^\ell})$ respectively.

We next notice that $M_{p^\ell,n}$ is $M_{p^\ell,1}$ tensored with itself $n$ times which we denote as $M_{p^\ell,n}=M_{p^\ell,1}^{\otimes n}$. Using $M_{p^\ell,1}=\overline{L}_{p^\ell}\overline{D}_{p^\ell}$ and Fact \ref{fact:multOfKronecker} we have $M_{p^\ell,n}=M_{p^\ell,1}^{\otimes n}=\overline{L}_{p^\ell}^{\otimes n}\overline{D}_{p^\ell}^{\otimes n}$. As $L_{p^\ell}$ was invertible with its inverse also having entries in $\Z[z]$ we see that $\overline{L}_{p^\ell}$ is also invertible and $\left(\overline{L}_{p^\ell}^{\otimes n}\right)^{-1}M_{p^\ell,n}= \overline{D}_{p^\ell}^{\otimes n}$. Using Lemma \ref{lem:crankMatMult} we have that,
$$\rank_{\F_p} M_{p^\ell,n}\ge \rank_{\F_p} \overline{D}_{p^\ell}^{\otimes n}.$$
As $\overline{D}_{p^\ell}$ is upper triangular so will $\overline{D}_{p^\ell}^{\otimes n}$ be. Therefore, to lower bound the rank of $\overline{D}_{p^\ell}^{\otimes n}$ we can lower bound the number of non-zero diagonal elements. 

The diagonal elements in $\overline{D}_{p^\ell}^{\otimes n}$ correspond to the product of diagonal elements chosen from $n$ copies of $\overline{D}_{p^\ell}$. Recall, the rows and columns of $\overline{D}_{p^\ell}$ are labelled by $j\in \{0,\hdots,p^\ell-1\}$ with
\begin{align*}
 D_{m}(j,j)=      \prod_{i=0}^{j-1} (z^j-z^i)
 \end{align*}
Setting $z-1=w$ we note that $\F_p[z]/\langle z^{p^\ell}-1\rangle$ is isomorphic to $\F_p[w]/\langle w^{p^\ell}\rangle$. $D_{p^\ell}(j,j)$ can now be written as
\begin{align*}
 D_{m}(j,j)=     (1+w)^{j(j-1)/2} \prod_{i=1}^{j} ((1+w)^i-1)
 \end{align*}
Using Lucas's Theorem (Theorem \ref{thm:luc}) we see that the largest power of $t$ which divides $(1+w)^l-1$ is the same as the largest power of $p$ which divides $l$. For any $j\le p^\ell-1$ , therefore the largest power of $w$ which divides $\overline{D}_{p^\ell}(j,j)$ is at most
\begin{align*}
    \sum\limits_{t=0}^{\lfloor \log_p(j)\rfloor} \left(\left\lfloor \frac{j}{p^t} \right\rfloor-\left\lfloor \frac{j}{p^{t+1}} \right\rfloor\right)p^t&=j+\sum\limits_{t=1}^{\lfloor \log_p(j)\rfloor} \left\lfloor \frac{j}{p^t} \right\rfloor p^{t-1}(p-1)\\
    &\le j(1+\lfloor \log_p(j)\rfloor(1-1/p))\\
    &\le j(\ell -(\ell-1)/p).\numberthis \label{eq:jBd}
\end{align*}
Consider the set of tuples $(j_1,\hdots, j_n)\in \N$ such that $j_1+\hdots+j_n\le \lceil p^\ell/\ell \rceil $. Using \eqref{eq:jBd} we see that the diagonal entry in $\overline{D}_{p^\ell}^{\otimes n}$ corresponding to the tuple will be divisible by at most $w^{\lceil p^\ell/\ell \rceil(\ell -(\ell-1)/p)}$. It is easy to check that $\lceil p^\ell/\ell \rceil(\ell -(\ell-1)/p)\le p^\ell -1$ which will guarantee that the $(j_1,\hdots,j_n)$'th diagonal entry of $\overline{D}_{p^\ell}^{\otimes n}$ is non-zero. 

This gives us at least $\binom{\lceil p^\ell \ell^{-1} \rceil+n}{n}$ non-zero diagonal entries proving the desired rank bound.
\end{proof}

%%% AUTHOR: optional acknowledgments here
% \section*{Acknowledgments} %%  you may comment this out if no Ackno
% The authors are grateful to the anonymous reviewers for finding
% a bug in the main result.

%%% AUTHOR:
%%% Bibliography goes here. Note that the arXiv cannot process bibtex
%%% or biber bibliographies.  Example of acceptable bibliograpy format:
\bibliographystyle{amsplain}

%\bibliography{bibliography}
% \begin{thebibliography}{99}
% \bibitem{bergelson-johnson-moreira}
% Vitaly Bergelson, John H. Johnson Jr., and Joel Moreira.
% \newblock New polynomial and multidimensional extensions of classical partition
%   results.
% \newblock 2015, arXiv:1501.02408.

% \bibitem{cilleruelo}
% Javier Cilleruelo.
% \newblock Combinatorial problems in finite fields and {S}idon sets.
% \newblock {\em Combinatorica}, 32(5):497--511, 2012.

% \end{thebibliography}
%% AUTHOR: You can generate such a bibliography from a .bib file by 
%% running pdflatex/bibtex/pdflatex/pdflatex and then pasting the .bbl file
%% between \begin{thebibliography} and \end{bibliography}

%%% AUTHOR: Include a short description of each author following the
%%% structure below. Use the same short tags used previously.  
%%% Use \imageat{} and \imagedot{} instead of "@" and "." in
%%% email addresses-this replaces the symbols with graphics to avoid 
%%% e-mail address harvesting from the .pdf file
\begin{aicauthors}
\begin{authorinfo}[pgom]
  Manik Dhar\\
  Massachusetts Institute of Technology\\
  
  Cambridge, MA, USA\\
  dmanik\imageat{}mit\imagedot{}edu \\
  \url{https://dharmanik.github.io}
\end{authorinfo}
\end{aicauthors}

\end{document}